\documentclass[12pt,a4paper]{article} 
\usepackage{amsmath,amssymb,amsthm,amsfonts,amscd,euscript,verbatim, t1enc, newlfont}
\usepackage{hyperref}
 \theoremstyle{plain}
\newtheorem{theo}{Theorem}[subsection]

\newtheorem{pr}[theo]{Proposition}

 \newtheorem{lem}[theo]{Lemma}
 \newtheorem{coro}[theo]{Corollary}
  
\theoremstyle{remark}
\newtheorem{rema}[theo]{Remark}

\theoremstyle{definition}
\newtheorem{defi}[theo]{Definition}
\newtheorem*{notat}{Notation}

 \newcommand\lan{\langle}
\newcommand\ra{\rangle}
\newcommand\bl{\bigl(} \newcommand\br{\bigl)}

\newcommand\ob{^{-1}}

\newcommand\dmge{DM^{eff}_{gm}{}}

\newcommand\dmgm{DM_{gm}}

\newcommand\zol{{\mathbb{Z}[\frac{1}{l}]}}

\newcommand\mg{\mathcal{M}}

\newcommand\obj{Obj}
\newcommand\mo{Mor}
\newcommand\id{id}
\newcommand\cu{\underline{C}}
\newcommand\du{{\underline{D}}}
\newcommand\eu{{\underline{E}}}
\newcommand\au{\underline{A}}

\newcommand\hrt{{\underline{Ht}}}
\newcommand\hw{{\underline{Hw}}}

\newcommand\z{{\mathbb{Z}}}

\newcommand\ql{{\mathbb{Q}_l}}

\newcommand\q{{\mathbb{Q}}}

\newcommand\p{\mathbb{P}}

\newcommand\al{\alpha}
\newcommand\be{\beta}
\newcommand\gam{\gamma}
\newcommand\de{\delta}

\newcommand\ns{\{0\}}

\DeclareMathOperator\inli{\varinjlim}

\newcommand\chow{Chow}

\newcommand\tcho{t_{Chow}}
\newcommand\htcho{\underline{Ht}_{Chow}}

\newcommand\ab{Ab}

\newcommand\spe{\operatorname{Spec}\,}

\DeclareMathOperator\imm{\operatorname{Im}}
\DeclareMathOperator\co{\operatorname{Cone}}
\DeclareMathOperator\prt{Pre-Tr}

\DeclareMathOperator\adfu{\operatorname{AddFun}}

\newcommand\trp{Tr^+}

\newcommand\gds{\mathfrak{D}_s}

\newcommand\dms{DM(S)}
\newcommand\dmcs{DM_c(S)}
\newcommand\dmck{DM_c(K)}
\newcommand\dmcu{DM_c(U)}

\newcommand\dmx{DM(X)}
\newcommand\dmcx{DM_c(X)}
\newcommand\dm{DM}
\newcommand\dmc{DM_c}
\newcommand\dmy{DM(Y)}
\newcommand\dmcy{DM_c(Y)}
\newcommand\chows{Chow(S)}
\newcommand\hwchow{{\underline{Hw}}_{\chow}}

\newcommand\chowx{Chow(X)}
\newcommand\chowy{Chow(Y)}
\newcommand\wchow{{w_{Chow}}}
\newcommand\tchow{{t_{Chow}}}

\newcommand\wchowb{{w_{Chow}^{big}}}
\newcommand\hwchowb{{{\underline{Hw}}_{Chow}^{big}}}

\newcommand\ch{Chow}
\newcommand\sred{{S_{red}}}
\newcommand\sss{{\mathcal{S}}}

\newcommand\on{\mathcal{OP}}
\newcommand\op{\mathcal{ON}}
\newcommand\oz{\mathcal{OZ}}

\newcommand\onx{\mathcal{OP}(X)}
\newcommand\opx{\mathcal{ON}(X)}

\newcommand\onal{\mathcal{OP}(\al)}
\newcommand\opal{\mathcal{ON}(\al)}
\newcommand\onsl{\mathcal{OP}(S^\al_l)}
\newcommand\opsl{\mathcal{ON}(S^\al_l)}

\newcommand\ff{\mathbb{F}}
\newcommand\fq{\mathbb{F}_q}

\newcommand\dbm{D^b_m}
\newcommand\dsh{DSH}
\newcommand\dshslz{D^bSh^{et}(S',\ql)}
\newcommand\dshl{D^bSh^{et}(-,\ql)}
\newcommand\hetl{{\mathcal{H}}^{et}_{\ql}{}}

\newcommand\dhsx{D^bSh^{et}(X,\ql)}

\newcommand\dhsxz{D^bSh^{et}(X_0,\ql)}

\begin{document}

 \title{
Weights for relative motives; relation to mixed complexes of sheaves}
 \author{Mikhail V. Bondarko
   \thanks{ 
 The work is supported by RFBR
(grants no. 11-01-00588 and 12-01-33057) and by the Saint-Petersburg State University research grant no. 6.38.75.2011.}} 
\author{Mikhail V. Bondarko}
\maketitle
\begin{abstract} 
The main goal of this paper is to define the {\it Chow weight structure} $\wchow$ for the category $\dmcs$ of (constructible) Beilinson motives over any excellent separated  
finite-dimensional base scheme $S$
 (this is the version of
Voevodsky's motives over $S$ described by Cisinski and Deglise). We also study the functoriality properties of 
$\wchow$ (they are very similar to those for weights of mixed complexes of sheaves, as described in \S5 of \cite{bbd}).

 As shown in a preceding paper,   (the existence of) $\wchow$ automatically yields a certain exact conservative weight complex functor  $\dmcs\to K^b(\chows)$. Here $\chows$ is the {\it heart} of $\wchow$;  it is 'generated' by motives of regular schemes that are projective over $S$. We also obtain that $K_0(\dmcs)\cong K_0(\chows)$ (and define a certain 'motivic Euler characteristic' for $S$-schemes).

 Besides, we obtain (Chow)-weight spectral sequences and filtrations for any (co)homology of motives; we discuss their relation to Beilinson's 'integral part' of motivic cohomology and with weights of  mixed complexes of sheaves. For the study of the latter we also introduce a new formalism of {\it relative weight structures}.
\end{abstract}
  \tableofcontents

 \section*{Introduction}
The goal of  this paper is to prove (independently from \cite{hebpo})  that the {\it Chow weight structure} $\wchow$ (as introduced in \cite{bws} for Voevodsky's motives over a perfect field $k$) can also be defined for the category $\dmcs$ of motives with rational coefficients over any ('reasonable') base scheme $S$ (in \cite{degcis}, where this category was described and studied, $\dmcs$ was called the category of Beilinson motives; one can also consider the 'big' category of $S$-motives $\dms\supset \dmcs$ here). The {\it heart}  $\hwchow$ of $\wchow$ is 'generated' by the motives of regular schemes that are projective over $S$ (tensored by $\q(n)[2n]$ for all $n\in \z$). We also study the functoriality properties of $\wchow$
(they are very similar to the functoriality of weights for mixed complexes of sheaves, as described in \S5 of \cite{bbd}).

As was shown in \cite{bws}, the existence of $\wchow$ yields several nice consequences.
In particular, there exists  a {\it weight complex} functor $t:\dmcs\to K^b(\chows)$, as well as  {\it Chow-weight} spectral sequences and  filtrations, and {\it virtual $t$-truncations} for any (co)homological functor $H:\dmcs\to \au$.

We also relate the weights for $S$-motives to the 'integral part' of motivic cohomology (as constructed in \cite{scholl};
cf. \S2.4.2 of \cite{bei85}), and with the  weights of mixed complexes of sheaves (as defined in \cite{bbd} and in \cite{huper}). In order to study the latter we  introduce a new formalism of {\it relative weight structures}.

Besides, we prove  that $K_0(\dmcs)\cong K_0(\chows)$, and define a certain 'motivic Euler characteristic' for  $S$-schemes.

Now we (try to) explain why the concept of a weight structure is important for motives. Recall  that weight structures are natural counterparts of $t$-structures for triangulated categories; they allow  'decomposing' objects of a triangulated $\cu$ into Postnikov towers whose 'factors' belong to the {\it heart} $\hw$ of $w$.
Weight structures  were introduced in \cite{bws} (and independently in \cite{konk}).  They were thoroughly studied and applied to motives (over perfect fields) in \cite{bws}; in \cite{bger}  a {\it Gersten} weight structure on a certain category $\gds\supset \dmge(k)$ was constructed; see also the survey preprint \cite{bsurv}.

The Chow weight structure yields certain weights for any (co)homology of motives.  Note here: 'classical' methods of working with motives often fail (at our present level of knowledge) since they usually depend on (various) 'standard' motivic conjectures.
 In particular,  the 'classical' way to define  weights for a motif $M$ is to construct a motif $M_s$ such that $H^i(M_s)\cong W_{s}H^i(X)$  (for all $i\in \z$ and a fixed $s$; here $H$ is either singular or \'etale cohomology, and $W_{s}(-)$ denotes the $s$-th level of the weight filtration for the corresponding mixed Hodge structure or mixed Galois module). It is scarcely possible to do this without constructing 
 a certain motivic $t$-structure on $DM(-)$ (whose existence is very much conjectural at the moment). 
 For instance, in order to find such $M_s$ for  motives of smooth projective varieties one requires the so-called Chow-Kunneth decompositions; hence this is completely out of reach at our present level of knowledge.

The usage of weight structures  (for motives) allows one to avoid these difficulties completely. To this end instead of $H^i(M_s)$ one considers $\imm(H^i(\wchow_{\ge -s-i}M)\to H^i(\wchow_{\ge -s-i-1}M))$ 
(this is the corresponding virtual $t$-truncation of $H$ applied to $M$; see \S\ref{svirt} below). 
 Here $\wchow_{\ge r}M$ for $r\in \z$ are certain motives which can  
 be (more or less) explicitly described in terms of $M$; note in contrast that there are no general conjectures that allow  constructing the motivic $t$-truncations and Chow-Kunneth decompositions explicitly.
Whereas this approach is somewhat 'cheating' for pure motives (since it usually gives no new information on their cohomology),  it yields several interesting results on mixed motives and their (co)homology. The first paper somewhat related to this approach is  \cite{gs} (its main result was generalized in \cite{sg}); there a weight complex functor that is essentially a (very) particular case of 'our' one was introduced (and related to the cohomology with compact support of varieties).

Another example when constructions naturally coming from weight structures  yield interesting results is described in \S\ref{schws} below.

 Now we 
discuss the relation of our paper to some other articles on relative motives.

 This text was written independently from the recent article \cite{hebpo} (that appeared somewhat earlier) on the same subject.
 We describe  similarities and distinctions between H\'ebert's paper and  our paper.
 In the current paper we introduce two different constructions $\wchow$. The first one (see Theorem \ref{twchow}) is based on Theorem
 4.3.2(II) of \cite{bws}; the same method was also used in (Theorem 3.3 of) \cite{hebpo} (so, the plans of the proof coincide).
  In order to apply it one has to specify $\hwchow$ and prove that it is 'negative' in $\dmcs$ (i.e., certain morphism groups are zero). Now, our description of $\hwchow$ is just slightly distinct from the one of 
 ibid. The two proofs of negativity are substantially different. Our version of the statement is Lemma \ref{l4onepo}(I.1); in the opinion of the author its proof is simpler than the one of the (parallel) Theorem 3.2 of \cite{hebpo}. Our statement also 
 contains a formula for certain non-zero (and 'interesting') morphism groups. On the other hand, loc. cit. has a serious advantage over our result: in the notation of our Lemma, it does not require $f$ and $g$ to be quasi-projective. As a result,  Theorem 3.7 of \cite{hebpo} describes the functoriality properties of $\wchow$ with respect to not necessarily quasi-projective morphisms in contrast to our Theorem \ref{tfunctwchow}(II) (whereas the proofs of these two theorems are quite similar; another distinction is that we do not study the behaviour of $\wchow$ under tensor products and inner homomorphisms).
 
The second definition of $\wchow$ is given in \S\ref{schowunr}; it has no analogue in \cite{hebpo} (and also in the papers that we mention below; yet note that the methods of the proof of the  functoriality properties 
of the 'first' construction of $\wchow$ easily yield the proof of  Proposition \ref{pfunctwchowa} also). The idea of the second definition is to start with the 'classical' definition of Chow motives over perfect fields, and 'glue' the corresponding weight structures into a weight structure 
 over an arbitrary excellent separated finite-dimensional $S$ (that is not necessarily reasonable in the sense of Definition \ref{dreas} below). Since the residue fields of $S$ do not have to be perfect, in \S\ref{schowunr} we have to consider compositions of smooth projective morphisms with finite universal homeomorphisms. The advantage of the gluing construction is that it works over  a wider class of base schemes; on the other hand it does not yield an explicit description of (the whole) $\hwchow$. Note also that both smooth projective morphisms and finite universal homeomorphisms yield (higher) \'etale direct image functors that respect the local constancy of sheaves; so their consideration is related to the study of 'explicit constructibility' for the \'etale (co)homology of motives.

 The author should also note that he would have probably not noticed that the category $\chows=\hwchow$ has a nice description (over a reasonable $S$) if not for the papers \cite{haco} and \cite{sg}. In \cite{haco} the definition of  Chow motives over $S$ was given as a part of a large program of study of relative motives and intersection cohomology of varieties (that relies on several hard 'motivic' conjectures). In \cite{sg} certain analogues of (our) Chow motives were used in order to define (a sort of) weight complexes for $S$-schemes (only for one-dimensional $S$; cf. \S\ref{swc} below). Yet note that 
 these two articles do not treat (any) triangulated categories of 'mixed' motives over $S$;  
 hence it is difficult to apply them to cohomology of 'general' (finite type) $S$-schemes.

This paper (also) benefited from \cite{scholfcohom}. In ibid. a 'mixed motivic' description of Beilinson's 'integral part' of motivic cohomology (as constructed in \cite{scholl}; see also \S2.4.2 of \cite{bei85}) was proposed. The formulation of the main result of \cite{scholfcohom} uses the so-called intermediate extensions of mixed motives; so it heavily relies on the (conjectural!) existence of a 'reasonable' motivic $t$-structure on $\dmcs$; note that we describe an alternative construction of this 'part' that does depend on any conjectures (in \S\ref{schws} below; our construction vastly generalizes the one Scholl).

Lastly, we note that several of the results of the current paper were applied in \cite{bmm}, where the existence of the motivic $t$-structure on $\dmcs$ and of certain 'weights' for its heart were reduced to (certain) standard motivic conjectures over universal domains. Besides, analogues of some of the results of the current paper for motives with integral coefficients were proved in 
\cite{brmz}.

Now we list the contents of the paper. More details can be found at the beginnings of sections.

 In \S\ref{sprem} we recall the basic properties of Beilinson motives and weight structures. Most of the statements of the section are contained in \cite{degcis} and \cite{bws}; yet we also prove some new results.

 In \S\ref{swchow} we define the category $\chows$ of Chow motives over $S$ (related definitions can be found in \cite{haco}, \cite{hebpo}, and \cite{sg}). By definition, $\chows\subset \dmcs$; since $\chows$ is also negative in it and generates it (if $S$ is 'reasonable') we immediately obtain (using Theorem 4.3.2 of \cite{bws}) that there exists a weight structure $\wchow$ on $\dmcs$ whose heart is $\chows$. Next we study the 'functoriality' of $\wchow$ (with respect to the functors of the type $f^*,f_*,f^!,f_!$, for $f$ being a quasi-projective or, more generally, a 
 {\it smoothly embeddable} morphism of schemes).
  Our functoriality statements are parallel to the 'stabilities' 5.1.14 of \cite{bbd} (we 'explain' this similarity in the succeeding section). We also prove that Chow motives can be 'lifted from open subschemes up to retracts'; this statement could be called (a certain) 'motivic resolution of singularities'.
 Next we prove that $\wchow$ can be described 'pointwisely' (cf. \S5.1.8 of \cite{bbd}). Besides, we describe an alternative method for the construction of $\wchow$ (over arbitrary excellent separated finite-dimensional schemes; these don't have to be   reasonable). This method uses stratifications and 'gluing' of weight structures; it makes this part of the paper
  somewhat parallel to the study of weights of mixed complexes of sheaves  in \S5 of \cite{bbd}.

\S\ref{sapcoh} is dedicated to the applications of our main results.
 The existence of $\wchow$ automatically yields the existence of a conservative exact {\it weight complex} functor $\dmcs\to K^b(\chows)$, and the fact that $K_0(\dmcs)\cong K_0(\chows)$. We also  define a certain 'motivic Euler characteristic' for $S$-schemes.

 Next we recall 
 that $\wchow$ yields  functorial {\it Chow-weight} spectral sequences and filtrations.
 A very particular case of Chow-weight filtrations yields Beilinson's 'integral part' of motivic cohomology. 
 Chow-weight spectral sequences yield the existence of weight filtrations for 'perverse \'etale homology' of motives over finite type $\q$-schemes (this is not automatic for mixed perverse sheaves in this setting). We study in more detail the perverse \'etale homology of motives when $S=X_0$ is a variety over a finite field $\fq$. It is well known that mixed complexes of sheaves start to behave better if we extend scalars from $\fq$ to $\ff$ (this is the algebraic closure of $\fq$), i.e., pass to sheaves over $X=X_0\times_{\spe \fq}\spe \ff$. We (try to) axiomatize this situation and introduce the concept of a {\it relative weight structure}. Relative weight structures have several properties that are parallel to properties of 'ordinary' weight structures. The category $\dbm(X_0)$ (of mixed complexes of sheaves) possesses a relative weight structure whose heart is the class of (pure) complexes of sheaves of weight $0$. Besides, the \'etale realization functor $\dmc(X_0)\to \dbm(X_0)$ is {\it weight-exact}.

 In 
the Appendix 
 we recall the definition of a $t$-structure {\it adjacent} to a weight structure. Then we prove the existence of a (Chow)  $t$-structure $\tchow$ for $\dms$ that is adjacent to the Chow weight structure on it. We also study the functoriality of  $\tchow$ and relate it to   {\it virtual $t$-truncations} (for cohomological functors from $\dmcs$).

 The author is deeply grateful to prof. F. Deglise, prof. D. H\'ebert,  prof. M. Levine,  prof. I. Panin, and to the referees for their helpful comments. He would also like to express his gratitude to the officers and the guests of the Max Planck Institut f\"ur Mathematik, as well as to prof. M. Levine and to the Essen University for the wonderful working conditions during the work on this paper.

\begin{notat}

$\ab$ is the category of abelian groups.

For categories $C,D$ we write $C\subset D$ if $C$ is a full 
subcategory of $D$.

 For a category $C,\ X,Y\in\obj C$, we denote by
$C(X,Y)$ the set of  $C$-morphisms from  $X$ to $Y$.
We will say that $X$ is  a {\it
retract} of $Y$ if $\id_X$ can be factored through $Y$. Note: if $C$ is triangulated or abelian 
then $X$ is a  retract of
$Y$ if and only if $X$ is its direct summand.


 For an additive $D\subset C$ the subcategory $D$ is called
{\it Karoubi-closed}
  in $C$ if it
contains all retracts  of its objects in $C$. The full subcategory of $C$ whose objects are all retracts of objects of $D$ (in $C$) will be called the {\it Karoubi-closure} of $D$ in $C$.

$M\in \obj C$ will be called compact if the functor $C(M,-)$
commutes with all small coproducts that exist in $C$. In this paper (in contrast with the previous ones) we will only consider compact objects in those categories that are closed with respect to arbitrary small coproducts. 

$\cu$ below will always denote some triangulated category; usually it will
be endowed with a weight structure $w$ (see Definition \ref{dwstr}
below).

We will use the term 'exact functor' for a functor of
triangulated categories (i.e.,  for a functor that preserves the
structures of triangulated categories). 
We will call a covariant (resp. contravariant)
additive functor $H:\cu\to \au$ for an abelian $\au$
 {\it homological} (resp. {\it cohomological}) if
it converts distinguished triangles into long exact sequences.

For $f\in\cu (X,Y)$, $X,Y\in\obj\cu$, we will call the third vertex
of (any) distinguished triangle $X\stackrel{f}{\to}Y\to Z$ a cone of
$f$; recall that distinct choices of cones are connected by
(non-unique) isomorphisms.

We will often specify a distinguished triangle by two of its
morphisms.

 For a set of
objects $C_i\in\obj\cu$, $i\in I$, we will denote by $\lan C_i\ra$
the smallest strictly full triangulated subcategory containing all $C_i$; for
$D\subset \cu$ we will write $\lan D\ra$ instead of $\lan \obj
D\ra$. We will call the  Karoubi-closure of $\lan C_i\ra$ in $\cu$ the {\it triangulated category  generated by $C_i$}.

For $X,Y\in \obj \cu$ we will write $X\perp Y$ if $\cu(X,Y)=\ns$.
For $D,E\subset \obj \cu$ we will write $D\perp E$ if $X\perp Y$
 for all $X\in D,\ Y\in E$.
For $D\subset \cu$ we will denote by $D^\perp$ the class
$$\{Y\in \obj \cu:\ X\perp Y\ \forall X\in D\}.$$
Dually, ${}^\perp{}D$ is the class
$\{Y\in \obj \cu:\ Y\perp X\ \forall X\in D\}$.

We will say that some $C_i$, $i\in I$, {\it weakly generate} $\cu$ if for
$X\in\obj\cu$ we have: $\cu(C_i[j],X)=\ns\ \forall i\in I,\
j\in\z\implies X=0$ (i.e., if $\{C_i[j]\}^\perp$ contains only
 zero objects).

 $D\subset \obj \cu$  will be
called {\it extension-stable}
    if for any distinguished triangle $A\to B\to C$
in $\cu$ we have: $A,C\in D\implies
B\in D$.

We will call the smallest
Karoubi-closed  extension-stable subclass of $\obj
\cu$ containing $D$ the {\it envelope} of $D$.

Below all schemes will be
excellent separated 
of finite Krull dimension. 
Often our schemes will be {\it reasonable}; see Definition \ref{dreas} below.

 Morphisms of schemes 
by default will be  of finite type. 
We will say that $X/S$ is 
{\it smoothly embeddable} if it can be embedded into a smooth $X'/S$ (in particular, such an $X$ is quasi-finite over $S$).
Certainly, any quasi-projective $S$-scheme is smoothly embeddable 
(over $S$).

We will sometimes need certain stratifications of a scheme $S$. Recall that  a 
 stratification
 $\al$  is a presentation of  $S$ as $\cup S_l^\al$,
 where $S_l^\al$, $1\le l\le n$, are pairwise disjoint locally closed subschemes of $S$. Omitting $\al$, we will denote by
$j_l:S_l^\al\to S$  the corresponding immersions.
 We do not demand the closure of each $S_l^\al$ to be  the union of strata (though we could do this); we will only assume that 
 each $S_l^{\al}$ is open in $\cup_{i\ge l} S_i^{\al}$. 
 
Below we will identify a Zariski point (of a scheme $S$) with the spectrum of its residue field.

 \end{notat}

\section{Preliminaries: relative motives and weight structures}\label{sprem}


In \S\ref{sbrmot}  we recall some of basic properties of Beilinson motives over $S$ (as considered in \cite{degcis}; we also deduce certain results that were not stated in ibid. explicitly).

In \S\ref{sws} we recall some basics of the theory of weight structures (as developed in \cite{bws}); we also prove some new lemmas on the subject.

\subsection{Beilinson motives (after Cisinski and Deglise)}\label{sbrmot}

We list some of the  properties of the triangulated categories of Beilinson motives
(this is the version of relative
Voevodsky's motives with rational coefficients described by Cisinski and Deglise).
Sometimes we will need the following restriction on schemes.

\begin{defi}\label{dreas}
We will call a separated scheme $S$ {\it reasonable} if there exists an excellent 
separated
scheme $S_0$ of dimension lesser than or equal to $2$ 
such that $S$ is 
of finite type over $S_0$.
\end{defi}

\begin{pr}\label{pcisdeg}

Let $X,Y$ be  (excellent separated finite dimensional) 
schemes; $f:X\to Y$ is a  
finite type morphism. 
\begin{enumerate}

\item\label{imotcat} For any  $X$ a tensor triangulated $\q$-linear category $\dmx$ with the unit object $\q_X$ is defined; it is closed with respect to arbitrary small coproducts.

$\dmx$ is the category of  {\it Beilinson motives} over $X$, as described (and thoroughly studied) in \S14--15 of \cite{degcis}.

\item\label{imotgen}
The (full) subcategory $\dmcx\subset \dmx$ of compact objects is tensor triangulated, and  $\q_X\in \obj \dmcs$. $\dmcx$ weakly generates $\dmx$.

\item\label{iidcompl} All $\dmx$ and $\dmcx$ are idempotent complete.

\item\label{imotfun}  For any  $f$ 
the following functors
 are defined:
$f^*: \dm(Y) \leftrightarrows \dmx:f_*$ and $f_!: \dmx \leftrightarrows \dmy:f^!$; $f^*$ is left adjoint to $f_*$ and $f_!$ is left adjoint to $f^!$.

We call these the {\bf motivic image functors}.
Any of them (when $f$ varies) yields a  $2$-functor from the category of 
(separated finite-dimensional excellent) schemes
with  morphisms of finite type to the $2$-category of triangulated categories.
Besides,
all motivic image functors preserve compact objects (i.e., they could be restricted to the subcategories $\dmc(-)$); they also commute with arbitrary (small) coproducts. 

\item\label{iexch} 
For a Cartesian square
of finite type 
morphisms
$$\begin{CD}
Y'@>{f'}>>X'\\
@VV{g'}V@VV{g}V \\
Y@>{f}>>X
\end{CD}$$
we have $g^*f_!\cong f'_!g'{}^*$ and $g'_*f'{}^!\cong f^!g_*$.


\item\label{iupstar}  $f^*$ is symmetric monoidal; $f^*(\q_Y)=\q_X$.

\item\label{itate}  For any $X$ there exists a
Tate object $\q(1)\in\obj\dmcx$; tensoring by it yields an  exact Tate twist functor $-(1)$ on $\dmx$.
This functor is an auto-equivalence of $\dmx$; we will denote the inverse functor by $-(-1)$.

 Tate twists
commute with all motivic image functors mentioned (up to an isomorphism of functors).

Besides, for $X=\p^1(Y)$ there is a functorial isomorphism $f_!(\q_{\p^1(Y)})\cong \q_Y\bigoplus \q_Y(-1)[-2]$.

\item \label{ipur}

$f_*\cong f_!$ if $f$ is proper;
$f^!(-)\cong f^*(-)(s)[2s]$ 
 if $f$ is smooth  
 (everywhere) of relative dimension $s$. 

If $f$ is an open immersion, we just have $f^!=f^*$.

\item \label{ipura} If $i:S'\to S$ is an 
 immersion of regular schemes everywhere of codimension $d$, then $\q_{S'}(-d)[-2d]\cong i^!(\q_S)$.

\item\label{iglu}

If $i:Z\to X$ is a closed immersion, $U=X\setminus Z$, $j:U\to X$ is the complementary open immersion, then
the motivic image functors yield a {\it gluing datum} for $\dm(-)$  (in the sense of \S1.4.3 of \cite{bbd}; see also Definition 8.2.1 of \cite{bws}). That means that (in addition to the adjunctions given by assertion \ref{imotfun}) the following statements are valid.

(i)  $i_*\cong i_!$ is a full embeddings; $j^*=j^!$ is isomorphic to the
localization (functor) of $\dmx$ by
$i_*(\dm(Z))$.

(ii) For any $M\in \obj \dmx$ the pairs of morphisms $j_!j^!(M) \to M
\to i_*i^*(M)$ and $i_!i^!(M) \to M \to j_*j^*(M)$ can be completed to
distinguished triangles (here the connecting
morphisms come from the adjunctions of assertion \ref{imotfun}).

(iii) $i^*j_!=0$; $i^!j_*=0$.

(iv) All of the adjunction transformations (i.e., the corresponding units and counits) $i^*i_*\to 1_{\dm(Z)}\to
      i^!i_!$ and $j^*j_*\to 1_{\dm(U)}\to j^!j_!$ are isomorphisms of
      functors.

\item\label{igluc} For the subcategories $\dmc(-)\subset \dm(-)$
the obvious analogue of the previous assertion is fulfilled.


\item \label{itr}
If $f$ is a finite universal homeomorphism, $f^*$ is an equivalence of categories.

 \item\label{igenc}

 If $S$ is reasonable (see Definition \ref{dreas}), $\dmcs$ (as a triangulated category) is generated by $\{ g_*(\q_X)(r)\}$, where $g:X\to S$ runs through all projective morphisms  (of finite type) such that $X$ is regular, $r\in \z$.  

\item\label{icont}
Let $S$ be a scheme which is the limit of an essentially affine (filtering) projective system of  schemes $S_\be$ (for $\be\in B$). Then $\dmcs$ is isomorphic to the $2$-colimit 
 of the categories $\dmc(S_\be)$; in these isomorphisms all the connecting functors are given by the corresponding motivic inverse image functors (cf.  Remark \ref{ridmot}(2) below).

 \item\label{ibormo}
 If $X$ is regular (everywhere) of dimension $d$, $i:Z\to X$ is a closed embedding, $p,q\in \z$, then $\dmx(\q_X,i_!i^!(\q_Z)[p](q)) \cong \dm(Z)(\q_Z,i^!(\q_X)(q)[p])$
 is isomorphic to
$\ch_{d-q}(Z,2q-p)\otimes\q$ (which we define  as $Gr_{d-q}^{\gamma}K'_{2q-p}(Z)\otimes\q)$.  In particular, this morphism group is zero   if $p>2q$.

\end{enumerate}

\end{pr}
\begin{proof}
Almost all of these properties of Beilinson motives
are stated in (part C of) the Introduction of ibid.; the proofs are mostly contained in \S1, \S2, \S14, and \S15 of ibid. 

So, we will only prove those assertions that are not stated in ibid. (explicitly).

For (\ref{iidcompl}):  Since $\dmx$  is closed with respect to arbitrary small coproducts, it is idempotent complete by Proposition 1.6.8 of \cite{neebook}. Since a retract of a compact object is compact also, $\dmcx$ is also idempotent complete.

Since $i^!=i^*$ if $i$ is an open immersion, and $i^*(\q_S)=\q_{S'}$, it suffices to prove (\ref{ipura}) for $i$ being a closed immersion. In this case it is exactly Theorem 4 of \cite{degcis}.


We should also prove (\ref{igluc}). Assertion \ref{iglu} immediately yields everything except the fact that the (categoric) kernel of $j^*:\dmcx\to \dmc(U)$ is contained in $i_*(\dmc(Z))$. So, we should prove that $i_*(\obj \dm(Z))\cap \obj \dmcx=i_*(\obj \dmc(Z))$. This is easy, since $i^*i_*\cong 1_{\dm(Z)}$ and $i_*i^*$ preserves compact objects.


Assertion \ref{itr} 
is given by Proposition 2.1.9  of \cite{degcis} (note that we can apply the result cited by 
Theorem 14.3.3 of ibid.). 

Assertion \ref{igenc} is immediate from Proposition 
15.2.3 of ibid. 

It remains to prove (\ref{ibormo}). 
Combining (13.4.1.3) and Corollary 14.2.14  of ibid., we obtain that the groups in question are isomorphic to the $q$-th factor of the $\gamma$-filtration on $K^Z_{2q-p}(X)\otimes\q$ (of the $K$-theory of $X$ with support in $Z$). 
By Theorem 7 of \cite{souoper},
 this is the exactly the 
$d-q$-th
factor of the 
 $\gamma$-filtration of $K'_{2q-p}(Z)\otimes\q$.

\end{proof}

\begin{rema}\label{ridmot}

1. In \cite{degcis} for a smooth $f:X\to Y$  the object $f_!f^!(\q_Y)$ was denoted by $\mg_Y(X)$ (cf. also Definition 1.3 of \cite{scholfcohom}; yet note that in loc. cit. cohomological motives are considered, this interchanges $*$ with $!$ in the notation for motivic functors).
 We will not usually need this notation below (yet cf. Remarks \ref{rtwchow}(1) and \ref{rintel}(4)).


 2. In \cite{degcis} the functor $f^*$
  was constructed for any 
   morphism $f$ not necessarily of finite type; it  preserves compact objects (see 
   Theorem 15.2.1(1) of ibid.). Besides, for such an $f$ and any 
   finite type $g:X'\to X$ we have an isomorphism $f^*g_!\cong g'_!f'{}^*$ (for the corresponding $f'$ and $g'$; cf. part \ref{iexch} of the proposition).

 Below the only morphisms of infinite type that we will be interested in are limits of immersions (more precisely, for a Zariski point $K$  
 of a scheme $S$  we will consider the natural morphism $j_K:K\to S$; cf. Notation).

 Now note: if $f$ is a pro-open immersion, then one can define $f^!=f^*$.
 So, one can also define $j_K^!$ that preserves compact objects (cf. also \S2.2.12 of \cite{bbd}). The system of these functors satisfy the second assertion in part \ref{iexch} of the proposition (for a finite type 
  $g$).

3. If $f$ is a finite universal homeomorphism, then 
$f^*$ is an equivalence of categories by assertion \ref{itr} of our proposition. Hence its right adjoint $f_*$ is an equivalence also. Since $f^!$ is right adjoint to $f_!=f_*$, we conclude that $f_!$ and $f^!$ are equivalences too.

Similarly we obtain that $f^*\q_Y\cong f^!\q_Y\cong \q_X$ and $f_*\q_X= f_!\q_X\cong \q_Y$.


4. Most of the properties of Beilinson motives (as listed above) also hold for various 'sheaf-like' categories. In particular, the methods of the current paper could probably be used in order to prove
the existence of the weight structure $w$ constructed in Proposition 2.3(I) of \cite{btrans} for M. Saito's mixed Hodge modules (see  \cite{sai}). Yet the properties of mixed Hodge modules listed in \S1 of ibid. yield the existence of $w$ immediately. 

5. A nice concise exposition of the properties of Beilinson motives (that also follows \cite{degcis}) 
can  be found in \S2 of \cite{hebpo}.

\end{rema}

The following statements were not proved in \cite{degcis} explicitly; yet they follow from Proposition \ref{pcisdeg} easily. Below we will mostly need assertion I.1 in the case where
$g$ is
 projective; note that in this case
$g_*(\q_Y)\cong  g_!(\q_Y)$.

\begin{lem}\label{l4onepo}

I.1. 
Let $f:X\to S$ and $g:Y\to S$ be 
smoothly embeddable morphisms (see the Notation), $r,b,c\in \z$.
Assume  $X$ is connected and $Y$ is connected and regular.

Then   there exists an integer $e$ such that   $\dms(f_!(\q_X)(b)[2b],g_*(\q_Y)(c)[r+2c])\cong CH_{e}(X\times_S Y,-r)$ (cf. Proposition \ref{pcisdeg}(\ref{ibormo}) for the definition of the latter).
 In particular, $f_!(\q_X)(b)[2b]\perp g_*(\q_Y)(c)[r+2c]$ if $r>0$.

2. Let $i:S'\to S$ be an 
immersion of regular schemes everywhere of codimension $d$; let $g$ be smooth. 
 Then for $Y'=Y_{S'}$ and $g'=g_{S'}$ we have $i^!g_*(\q_Y)\cong g'_*(\q_{Y'})(-d)[-2d]$.

II Let $S=\cup S_l^\al$ be a stratification. Then for any $M,N\in \obj \dms$ there exists a filtration of $\dms(M,N)$  whose factors are certain subquotients of $\dm(S_l^\al)(j_l^*(M), j_l^!(N))$.

\end{lem}

\begin{proof}

I.1. By   Proposition \ref{pcisdeg}(\ref{itate}), we can assume that $c=0$. 

Next, we have $\dms(f_!(\q_X),g_*(\q_Y)(c)[r])\cong \dm(Y)(g^*f_!(\q_X),\q_Y(c)[r+2c])$ since $g^*$ is left adjoint to $g_*$. Applying Proposition \ref{pcisdeg}(\ref{iexch}), we obtain that the group in question is isomorphic to $$\dm(Y)(f'_!g'{}^*(\q_X),\q_Y(c)[r+2c])=\dm(Y)(f'_!(\q_{X\times_SY}),\q_Y(c)[r+2c])$$ (here $f'=f_Y$).

We denote $X\times_SY$ by $Z$.
Let $P$ be a smooth 
 $Y$-scheme containing  $Z$ as a closed subscheme;  we denote by $i:Z\to P$ and $p:P\to Y$ the corresponding morphisms. We can assume that $P$ is equidimensional and  everywhere of some dimension $d$ over $Y$. 

Then we have $$\dm(Y)(f'_!(\q_Z),\q_Y(c)[r])=\dms(p_!i_!(\q_Z),\q_Y(c)[r+2c]) \cong \dm(P)(i_!(\q_Z),p^!(\q_Y)(c)[r+2c])$$ (here we apply the adjunction of $p_!$ with $p^!$). By Proposition \ref{pcisdeg}(\ref{ipur}), the group in question is isomorphic to  $\dm(P)(i_!(\q_Z),p^*(\q_Y)(d+c)[r+2d+2c])\cong \dm(P)(i_!(\q_Z),\q_P(d+c)[r+2d+2c])$.  It remains to apply  Proposition \ref{pcisdeg}(\ref{ibormo}). 

2. $i^!g_*(\q_Y)\cong g'_*i'{}^!(\q_Y)$  by Proposition \ref{pcisdeg}(\ref{iexch})  (here $i'=i_Y$). So using Proposition \ref{pcisdeg}(\ref{ipura}) we obtain
the result.

II We prove the statement by induction on the number of strata.
By definition (see the Notation section) 
 $S_1^\al$ is open in $S$, and the remaining $S_l^\al$  yield a stratification of $S\setminus S_1^\al$.
We denote $S\setminus S_1^\al$ by $Z$, the (open) immersion $S_1^\al \to S$ by $j$ and the (closed) immersion $Z\to S$ by $i$.

Now we apply Proposition \ref{pcisdeg}(\ref{iglu}).
We obtain  a (long) exact sequence
$\dots\to \dms (i_*i^*(M),N)\to \dms(M,N)\to \dms (j_!j^!(M),N)
\to\dots$. The adjunctions of functors yield
$\dms(i_*i^*(M),N)\cong \dm(Z)(i^*(M), i^!(N))$ and $\dms(j_!j^!(M),N')\cong  \dm({S_1^\al})(j^*(M), j^!(N))$.

Now, by the inductive assumption the group $\dm(Z)(i^*(M), i^!(N))$ has a filtration whose factors are certain subquotients of $\dm(S_l^\al)(j_l^*(M), j^!(N))$ (for $l\neq 1$). This concludes the proof.

\end{proof}

\begin{rema}\label{rdimI1}
The proof of assertion I.1 certainly yields (for an arbitrary $b$) that  $e=\dim P -d +b-c$. Besides, if $Y$ is a Jacobson equicodimensional scheme (it suffices to assume that $S$ satisfies this condition; it is fulfilled for any scheme that is of finite type over the spectrum of a field or over $\spe \z$)
then we have $e=\dim Y+b-c$; cf. \S1.2 of \cite{bphtp}.
\end{rema}

\subsection{Weight structures: short reminder}\label{sws}

\begin{defi}\label{dwstr}

I A pair of subclasses $\cu_{w\le 0},\cu_{w\ge 0}\subset\obj \cu$ 
will be said to define a weight
structure $w$ for $\cu$ if 
they  satisfy the following conditions:

(i) $\cu_{w\ge 0},\cu_{w\le 0}$ are additive and Karoubi-closed in $\cu$
(i.e., contain all $\cu$-retracts of their objects).

(ii) {\bf Semi-invariance with respect to translations.}

$\cu_{w\le 0}\subset \cu_{w\le 0}[1]$, $\cu_{w\ge 0}[1]\subset
\cu_{w\ge 0}$.

(iii) {\bf Orthogonality.}

$\cu_{w\le 0}\perp \cu_{w\ge 0}[1]$.

(iv) {\bf Weight decompositions}.

 For any $M\in\obj \cu$ there
exists a distinguished triangle
\begin{equation}\label{wd}
B\to M\to A\stackrel{f}{\to} B[1]
\end{equation} 
such that $A\in \cu_{w\ge 0}[1],\  B\in \cu_{w\le 0}$.

II The full category $\hw\subset \cu$ whose object  class is
$\cu_{w=0}=\cu_{w\ge 0}\cap \cu_{w\le 0}$ 
 will be called the {\it heart} of 
$w$.


III $\cu_{w\ge i}$ (resp. $\cu_{w\le i}$, resp.
$\cu_{w= i}$) will denote $\cu_{w\ge
0}[i]$ (resp. $\cu_{w\le 0}[i]$, resp. $\cu_{w= 0}[i]$).

IV We denote $\cu_{w\ge i}\cap \cu_{w\le j}$
by $\cu_{[i,j]}$ (so it equals $\ns$ for $i>j$).

V We will  say that $(\cu,w)$ is {\it  bounded}  if
$\cup_{i\in \z} \cu_{w\le i}=\obj \cu=\cup_{i\in \z} \cu_{w\ge i}$.

VI Let $\cu$ and $\cu'$ 
be triangulated categories endowed with
weight structures $w$ and
 $w'$, respectively; let $F:\cu\to \cu'$ be an exact functor.

$F$ will be called {\it left weight-exact} 
(with respect to $w,w'$) if it maps
$\cu_{w\le 0}$ to $\cu'_{w'\le 0}$; it will be called {\it right weight-exact} if it
maps $\cu_{w\ge 0}$ to $\cu'_{w'\ge 0}$. $F$ is called {\it weight-exact}
if it is both left 
and right weight-exact.

VII Let $H$ be a 
full subcategory of a triangulated $\cu$.

We will say that $H$ is {\it negative} if
 $\obj H\perp (\cup_{i>0}\obj (H[i]))$.

 VIII We call a category $\frac A B$ a {\it factor} of an additive
category $A$
by its (full) additive subcategory $B$ if $\obj \bl \frac A B\br=\obj
A$ and $(\frac A B)(M,N)= A(M,N)/(\sum_{O\in \obj B} A(O,N) \circ
A(M,O))$.

IX For an additive $B$ we will consider the  category
of 'formal coproducts' of objects of $B$:
its objects are (formal) $\coprod_{j\in J}B_j:\ B_j\in \obj B$, and
$\mo(\coprod_{l\in L}B_{l},\coprod_{j\in J}B'_{j})
=\prod_{l\in L}(\bigoplus_{j\in J}B(B_{l},B'_{j}))$;
here $L,J$ are index sets. We will call the idempotent completion of this category the {\it big hull} of $B$.

\end{defi}

\begin{rema}\label{rstws}

1. If $B$ is a full subcategory of compact objects in an additive $C$, and $C$ is idempotent complete and closed with respect to arbitrary small coproducts, then there exists a natural full embedding of the big hull of $B$ into $C$. Note here:  if $C$ is triangulated and closed with respect to arbitrary small coproducts, then it is necessarily idempotent complete (see Proposition 1.6.8 of \cite{neebook}).

2. A  simple (and yet useful) example of a weight structure comes from the stupid
filtration on the homotopy categories of cohomological complexes
$K(B)\supset K^b(B)$ for an arbitrary additive category $B$. 
In this case
$K(B)_{w\le 0}$ (resp. $K(B)_{w\ge 0}$) will be the class of complexes that are
homotopy equivalent to complexes
 concentrated in degrees $\ge 0$ (resp. $\le 0$).  The heart of this weight structure (either for $K(B)$ or for $K^b(B)$)
is the  the Karoubi-closure  of $B$
 in the corresponding category.  We will use the notation $K(B)_{[i,j]}$ below following Definition \ref{dwstr}(IV).

3. A weight decomposition (of any $M\in \obj\cu$) is (almost) never canonical; 
still we will sometimes denote (any choice of) a pair $(B,A)$ coming from in (\ref{wd}) by $(w_{\le 0}M,w_{\ge 1}M)$. 

For an $l\in \z$ we denote by $w_{\le l}M$ (resp. $w_{\ge l}M$) a choice of  $w_{\le 0}(M[-l])[l]$ (resp. of $w_{\ge 1}(M[1-l])[l-1]$).

 We will call (any choices of) $w_{\le l} M$ and $w_{\ge l}M$  {\it weight truncations} of $M$.
A certain illustration of this notation (in a more general context of relative weight structure) can be found in Proposition \ref{pbrw}(\ref{irffunctwd}) below.

4. In the current paper we use the 'homological convention' for weight structures; 
it was previously used in \cite{hebpo}, \cite{wildic},  
and  \cite{btrans}, whereas in 
\cite{bws} and in \cite{bger} the 'cohomological convention' was used. In the latter convention 
the roles of $\cu_{w\le 0}$ and $\cu_{w\ge 0}$ are interchanged, i.e., one considers   $\cu^{w\le 0}=\cu_{w\ge 0}$ and $\cu^{w\ge 0}=\cu_{w\le 0}$. So,  a complex $X\in \obj K(B)$ whose only non-zero term is the fifth one (i.e., $X^5\neq 0$) 
 has weight $-5$ in the homological convention, and has weight $5$ in the cohomological convention. Thus the conventions differ by 'signs of weights'; 
 $K(B)_{[i,j]}$ is the class of retracts of complexes concentrated in degrees $[-j.-i]$. 
  
 
\end{rema}

Now we recall those properties of weight structures that
will be needed below (and that can be easily formulated).
We will not mention more complicated matters (weight spectral sequences and weight complexes) 
here; instead we will just formulate
the corresponding 'motivic' results below.

\begin{pr} \label{pbw}
Let $\cu$ be a triangulated category. 

\begin{enumerate}

\item\label{idual}
$(C_1,C_2)$ ($C_1,C_2\subset \obj \cu$) define a weight structure on $\cu$ if and only if
$(C_2^{op}, C_1^{op})$ define a weight structure on $\cu^{op}$.

\item\label{iext} Let $w$  be a weight structure on
 $\cu$. Then  $\cu_{w\le 0}$, $\cu_{w\ge 0}$, and $\cu_{w=0}$
are extension-stable.

Besides, for any $M\in \cu_{w\le 0}$ we have $w_{\ge 0}M\in \cu_{w=0}$ (for any choice of $w_{\ge 0}M$).

\item\label{iort} Let $w$  be a weight structure on
 $\cu$. Then $\cu_{w\ge 0}=(\cu_{w\le -1})^\perp$ and
 $\cu_{w\le 0}={}^\perp \cu_{w\ge 1}$ (see Notation).

\item\label{iuni} 
Suppose that $v,w$ are weight structures for $\cu$; let $\cu_{v\le
0}\subset \cu_{w\le 0}$ and $\cu_{v\ge 0}\subset \cu_{w\ge 0}$.
Then $v=w$ (i.e., the inclusions are equalities).

\item\label{iwsidc} Let $w$ be a bounded weight structure on $\cu$. Then $w$ extends to a bounded weight structure for the idempotent completion $\cu'$ of $\cu$ (i.e., there exists a weight structure $w'$ for $\cu'$ such that the embedding $\cu\to \cu'$ is weight-exact); its heart is the idempotent completion of $\hw$.

\item \label{igen}
Assume   $H\subset \obj \cu$ is negative and   $\cu$ is idempotent complete. Then
there exists a unique weight structure $w$ on the Karoubi-closure $T$
of $\lan H\ra$ in $\cu$ such that $H\subset T_{w=0}$. Its heart is
the {\it envelope} (see the Notation) of $H$ in 
$\cu$; it is the idempotent completion of $H$ if $H$ is additive.

\item \label{iwgen} For the weight structure mentioned in the
    previous assertion, $T_{w\le 0}$ is the
     envelope of  $\cup_{i\le 0}H[i]$;
$T_{w\ge 0}$ is the envelope of  $\cup_{i\ge 0}H[i]$.

\item \label{iadjco}
A composition of left (resp. right) weight-exact functors is left (resp. right) weight-exact.

\item \label{iadj}
Let $\cu$ and $\du$ be triangulated categories endowed with weight structures $w$ and $v$, respectively. Let
$F: \cu \leftrightarrows \du:G$ be adjoint functors. Then $F$ is left weight-exact if and only if $G$ is right weight-exact.

\item\label{iwfun} Let $\cu$ and $\du$ be triangulated categories endowed with weight structures $w$ and $v$, respectively; let $w$ be bounded. Then an exact functor $F:\cu\to \du$  is left (resp. right) weight-exact if and only if $F(\cu_{w=0})\subset \du_{v\le 0}$ (resp. $F(\cu_{w=0})\subset \du_{v\ge 0}$).

\item\label{iloc}

Let $w$  be a weight structure on
 $\cu$; let $\du\subset \cu$ be a 
triangulated subcategory of
$\cu$. Suppose
 that $w$ yields a weight structure $w_{\du}$ for  $\du$
(i.e., $\obj \du\cap \cu_{w\le
 0}$ and $\obj \du\cap \cu_{w\ge
 0}$ give a weight structure on $\du$). 

 Then $w$ also induces a weight structure on
 $\cu/\du$ (the localization, i.e., the Verdier quotient of $\cu$
by $\du$) in the following sense: the Karoubi-closures of $\cu_{w\le
 0}$ and $\cu_{w\ge
 0}$ (considered as classes of objects of $\cu/\du$) give a weight structure $w'$ for $\cu/\du$
(note that $\obj \cu=\obj \cu/\du$). Besides, there exists a full embedding $\frac {\hw}{\hw_{\du}}\to \hw'$; $\hw'$ is the Karoubi-closure  of $\frac{\hw}{\hw_{\du}}$ in
$\cu/\du$.

\item \label{iwegen}
Suppose that $\du\subset \cu$ is a full category of compact objects endowed with bounded a weight structure $w'$. Suppose that $\du$ weakly generates $\cu$; let $\cu$  admit arbitrary (small) coproducts. Then $w'$ can be extended to a weight structure $w$ for $\cu$. Its heart is the big hull of $\hw'$ (as defined in  Definition \ref{dwstr}(IX)).

\item\label{igluws}
Let $\du\stackrel{i_*}{\to}\cu\stackrel{j^*}{\to}\eu$ be a part of a gluing datum. This means
that $\du,\cu,\eu$ are triangulated categories, $i^*$ and $j^*$ are exact functors;  $j^*$ is a localization functor, $i_*$ is an embedding of the categorical kernel of $j^*$ into $\cu$; $i_*$ possesses both a left adjoint $i^*$ and a right adjoint $i^!$ (see Chapter 9 of \cite{neebook}; note that this piece of a datum extends to a datum similar to that described in  Proposition \ref{pcisdeg}(\ref{iglu})).

Then for any pair of weight structures on $\du$ and $\eu$ (we will denote them by $w_\du$ and $w_\eu$, respectively)
there exists a 
 weight structure $w$ on $\cu$ such that both $i_*$ and $j^*$ are weight-exact (with respect to the corresponding weight structures). Besides, $i^!$ and $j_*$ are right weight-exact (with respect to the corresponding weight structures); $i^*$ and $j_!$ are left weight-exact. Moreover, 
$\cu_{w\ge 0}=C_1=\{M\in \obj
\cu:\ i^!(M)\in \du_{w_\du\ge 0} ,\ j^*(M)\in \eu_{w_\eu\ge 0} \}$ and
$\cu_{w\le 0}=C_2=\{M\in \obj \cu:\ i^*(M)\in \du_{w_\du\le 0} ,\ j^*(M)\in
\eu_{w_\eu\le 0} \}$. Lastly, $C_1$ (resp. $C_2$) is the  envelope  of $j_!(\eu_{w\le
0})\cup  i_*(\du_{w\le 0})$  (resp. of $ j_*(\eu_{w\ge 0})\cup i_*(\du_{w\ge 0})$).

\item\label{igluwsc} In the setting of the previous assertion, if $w_{\du}$ and $w_{\eu}$ are  bounded, then: $w$ is bounded also; besides, $\cu_{w\le 0}$ is the envelope of $\{i_*(\du_{w_{\du}=l}),\ j_!(\eu_{w_{\eu}=l}),\ l\le 0\}$;
$\cu_{w\ge 0}$ is the envelope of $\{i_*(\du_{w_{\du}=l}),\ j_*(\eu_{w_{\eu}=l}),\ l\ge 0\}$.

\item\label{igluwsn} In the setting of  assertion \ref{igluws}, the weight structure $w$ described is the only weight structure on $\cu$ such that both $i_*$ and $j^*$ are weight-exact.

\end{enumerate}
\end{pr}
\begin{proof}

Most of the assertions 
were proved in \cite{bws} (pay attention to Remark \ref{rstws}(4)!); 
see Remark 1.1.2(1), Proposition 1.3.3(3,6,1,2), Lemma 1.3.8, Proposition 5.2.2, Theorem 4.3.2(II) (together with  its proof), Remark 4.4.6, Proposition 8.1.1, Theorem 4.5.2, and Theorem 8.2.3 (together with Remark 8.2.4(1))
  of ibid., respectively.

We only have to prove assertions \ref{iadj}, \ref{iwfun}, \ref{igluwsc}, and 
\ref{igluwsn}.

(\ref{iadj}) follows immediately from assertion \ref{iort} (using the definition of adjoint functors).

(\ref{iwfun}) is immediate from assertion \ref{iwgen} by assertion \ref{iext}.

If $w_{\cu}$ and $w_{\du}$ are bounded, then $w$ also is by definition. The remaining part of assertion \ref{igluwsc} is immediate from Remark 8.2.4(1) of \cite{bws} and assertion \ref{iwgen}.

(\ref{igluwsn}): Suppose that the assumptions of assertion \ref{igluws} are fulfilled, and consider some weight structure $v$ for $\cu$ such that $i_*$ and $j^*$ are weight-exact.

Since $i_*$ and $j^*$ are weight-exact, by assertion \ref{iadj} we obtain:  $i^!$ and $j_*$ are right weight-exact; $i^*$ and $j_!$ are left weight-exact (with respect to the corresponding weight structures). 
Hence the class $\cu_{v\le 0}$ (resp. $\cu_{v\ge 0}$) is contained in $C_1$ (resp. in $C_2$) in the notation of assertion \ref{igluws}. Since the couple $(C_1,C_2)$ does yield a weight structure $w$ for $\cu$ (by loc. cit.), by assertion \ref{iuni} we obtain that $v=w$.

\end{proof}

\begin{rema}\label{rlift}

Part \ref{iloc} of the proposition can be re-formulated as follows. If $i_*:\du\to \cu$ is an embedding of triangulated categories that is weight-exact (with respect to certain weight structures for $\du$ and $\cu$), an exact functor $j^*:\cu\to \eu$ is equivalent to the localization of $\cu$ by $i_*(\du)$, then there exists a unique weight structure $w'$ for $\eu$ such that $j^*$ is weight-exact; $\hw_{\eu}$ is the Karoubi-closure of $\frac {\hw} {i_*(\hw_{\du})}$ (with respect to the natural functor $\frac {\hw}{i_*(\hw_{\du})}\to \eu$).

\end{rema}

\section{The Chow weight structure: two constructions and basic properties}\label{swchow} 

In \S\ref{stmain} we define the category $\chows$ of Chow motives over $S$ (similar definitions can be found in \cite{haco}, \cite{hebpo}, and \cite{sg}). By our definition, $\chows\subset \dmcs$; since $\chows$ is also negative in it and generates it (if $S$ is reasonable; here we use  the properties of $\dmcs$ proved in \S\ref{sbrmot}) we immediately obtain (by Proposition \ref{pbw}(\ref{iwgen})) that there exists a weight structure on $\dmcs$ whose heart is $\chows$.

In \S\ref{sfwchow} we study the 'functoriality' of $\wchow$ (with respect to functors of the type $f^*,f_*,f^!$, and $f_!$, for $f$ being a 
smoothly embeddable  morphism of schemes). Our functoriality statements are parallel to the 'stabilities' 5.1.14 of \cite{bbd}; we will explain this similarity in the next section.
We also prove that $\wchow$ can be described 'pointwisely' (similarly to \S5.1.8 of \cite{bbd}), and prove that it is 'continuous' (in a certain sense).

In \S\ref{schowunr}  we describe an alternative method for the construction of $\wchow$ for $\dmcs$ (for an excellent separated finite-dimensional scheme $S$ that is not necessarily reasonable). This method uses stratifications and 'gluing' of weight structures; this makes this part of the paper
very much parallel to the study of weights of mixed complexes of sheaves  in \S5 of \cite{bbd}. Actually, this method is the first one developed by the author
(it was first proposed in Remark 8.2.4(3) of \cite{bws}, that was in its turn inspired by \cite{bbd}). We prove that this alternative method yields the same result as the method of \S\ref{stmain} if $S$ is reasonable. This yields some  
new descriptions of $\wchow$ (in this case); see Remark \ref{rflex}(2).

\subsection{Relative Chow motives; 
the 'basic' construction of $w_{\chow}$}\label{stmain}

We define $\chows$ as the Karoubi-closure of $\{f_!(\q_X)(r)[2r]\}=\{f_*(\q_X)(r)[2r]\}$ in $\dmcs$; here $f:X\to S$ runs through all finite type projective 
 morphisms such that $X$ is regular, $r\in \z$.

Till \S\ref{schowunr} we will assume that all schemes that we consider are reasonable (see Definition \ref{dreas}).


\begin{theo}\label{twchow}

I There exists a (unique) weight structure $\wchow$ for $\dmcs$ 
whose heart is $\chows$.

II 
$\wchow(S)$ can be extended to a weight structure $\wchowb$ for the whole $\dms$.
$\hwchowb$ is the big hull of $\chows$ (as defined in   Definition \ref{dwstr}(IX); see Remark \ref{rstws}(1)).

\end{theo}
\begin{proof}

I By  Proposition \ref{pbw}(\ref{igen}) it suffices to verify that $\chows$ is negative and generates $\dmcs$. The negativity of $\chows$ is immediate from  Lemma \ref{l4onepo}(I). $\chows$ generates $\dmcs$  by   Proposition \ref{pcisdeg}(\ref{igenc}).

II Since $\chows$ generates $\dmcs$, and $\dmcs$ weakly generates $\dms$ (by Proposition \ref{pcisdeg}(\ref{imotgen})), $\chows$ weakly generates $\dms$. 

 Hence the assertion follows immediately from assertion I and Proposition \ref{pbw}(\ref{iwegen}). 

\end{proof}

\begin{rema}\label{rtwchow}
1. In particular, the theorem holds for $S$ being the spectrum of a (not necessarily perfect) field $k$.

 For a perfect $k$ the existence of $\wchow$  
 was already proved in \S6 
of \cite{bws}. 
Note here that $\dmc(\spe k)\cong \dmgm\q(k)$ for a perfect $k$ (in the notation of  loc. cit.; $\dmgm\q(k)$ denotes the category of motives with rational coefficients), whereas $p_!\q_P(r)[2r]$ yields a Chow motif over $k$ (for any $r\in \z$ and $p:P\to \spe k$ being a smooth projective morphism; recall here that the 'ordinary' category of Chow motives over $k$ can be fully embedded into $\dmgm$). 

2. Moreover (as was kindly pointed out by the referee) if $k^p$ is 
the perfect closure of an arbitrary field $k$, then $\dmc(\spe k^p)\cong \dmc(\spe k)$; see Proposition \ref{pcisdeg}(\ref{itr}). 

3. Besides, in \cite{mymot} a related differential graded 'description' of motives over a characteristic zero $k$ was given. It was generalized in \cite{lesm} to a description of the category of {\it smooth motives} over  $S$, when   $S$ is a smooth variety  over (a characteristic $0$ field) $k$; here the category of smooth motives is the triangulated category generated by motives of smooth projective $S$-schemes.

Note also: the restriction of $\wchow$ to smooth motives 
induces a weight structure for this category (that is coherent with the 'description' mentioned).

4. Our results would certainly look nicer if we  had a description of the composition of morphisms in $\chows$ (note here that  the morphism groups between 'generating objects' of $\chows$  can be immediately computed using  Lemma \ref{l4onepo}(I)).  The author conjectures that this composition is compatible with the ones described \S2 of \cite{haco} and in \S5.2 of \cite{sg}. 
 In order to prove this Levine's method could be quite useful, as well as the description of $\dms$ in terms of $qfh$-sheaves (see Theorem 16.1.2 of \cite{degcis}). Moreover, the methods of \cite{lesm} could possibly allow  giving a 'differential graded' description of the whole $\dmcs$ (extending the main result of ibid.). 

The author plans to study these matters further.

5. In Theorem 3.2 of \cite{hebpo} an orthogonality property (similar to that in Lemma \ref{l4onepo}(I.1)) was established for  not necessarily 
smoothly embeddable $f$ and $g$.
This yielded that $\{f_!(\q_X)(r)[2r]\}\in \dmcs_{w=0}$ for any proper (not necessarily projective!) $f$ such that $X$ is regular, $r\in \z$,
and allowed  generalizing Theorem \ref{tfunctwchow}(II) (below) to not necessarily 
smoothly embeddable morphisms.


Note also:  in our proof of 
Lemma \ref{l4onepo}(I.1) we actually only 
needed an embedding of $X\times_S Y$ into a regular finite type $S$-scheme; this also implies a certain generalization of  Theorem \ref{tfunctwchow}(II) (in particular, we can extend it to arbitrary finite type morphisms of  schemes that are of finite type over a fixed characteristic $0$ field). 

6. If $S$ is not reasonable, we still obtain that $\chows$ is negative. Hence, there exists a weight structure on $\lan \chows \ra$ whose heart is $\chows$ (since $\chows$ is idempotent complete). The problem is that we do not know whether $\lan \chows\ra$ is the whole $\dmcs$ (though this is true for motives over generic points of $S$, since those are reasonable; cf. also part 1 of this remark). 

One can also prove the existence of a certain analogue of the Chow weight structure over a not necessarily reasonable scheme $S$; see \S\ref{schowunr} below. The main disadvantage of this method is that it does not yield an 'explicit' description of  $\hwchow$ (though  $\hwchow\supset \chows$; cf. Remark \ref{rflex}(3)).

 \end{rema}

\subsection{Functoriality of $w_{\chow}$}\label{sfwchow}

Now we study (left and right) weight-exactness of the motivic image functors. These statements are very similar to the properties of pure complexes of constructible sheaves. This is not surprising; cf. \S\ref{sperv} below. 
In this subsection $S,X,Y$ (and hence also $Z$, $U$, and all $S_l^\al$) will be reasonable.

\begin{theo}\label{tfunctwchow}

I The functor $-(1)[2](=\otimes \q(1)[2])$ and its inverse $-(-1)[-2]:\dmcs\to\dmcs$ are weight-exact with respect to $\wchow$ for any 
$S$.

II Let $f:X\to Y$ be a 
 smoothly embeddable morphism of 
 schemes.

1. $f^!$ and $f_*$ are right weight-exact; $f^*$ and $f_!$ are left weight-exact.

2. Suppose moreover that $f$ is smooth. Then $f^*$
and $f^!$ are also weight-exact.

3. Moreover, $f^*$ is weight-exact for  $f$ being either (i) a finite universal homeomorphism or (ii) a
(filtering) projective limit of smooth morphisms such that the corresponding connecting morphisms are smooth 
 affine. In case (i) $f^!$ is weight-exact also.

III  Let $i:Z\to X$ be a closed immersion; 
let $j:U\to X$ be the complementary open immersion.

1. $\chow(U)$ is the idempotent completion of the factor (in the sense of  Definition \ref{dwstr}(VIII)) of $\chow(X)$ by $i_*(\chow(Z))$. 

2. For $M\in \obj \dmcx$ we have: $M\in \dmcx_{\wchow\ge 0}$ (resp. $M\in \dmcx_{\wchow\le 0}$) if and only if $j^!(M)\in \dmc(U)_{\wchow\ge 0}$ and $i^!(M)\in \dmc(Z)_{\wchow\ge 0}$ (resp.  $j^*(M)\in \dmc(U)_{\wchow\le 0}$ and $i^*(M)\in \dmc(Z)_{\wchow\le 0}$).

IV Let $S=\cup S_l^\al$ be a stratification, $j_l:S_l^\al\to S$ are the corresponding immersions. Then for $M\in \obj \dmcs$ we have: $M\in \dmcs_{\wchow\ge 0}$ (resp. $M\in \dmcs_{\wchow\le 0}$) if and only if $j_l^!(M)\in \dmc(S_l^\al)_{\wchow\ge 0}$ (resp. $j_l^*(M)\in \dmc(S_l^\al)_{\wchow\le 0}$)  for all $l$.

V 1. For any $S$ we have $\q_S\in \dmcs_{\wchow\le 0}$.

2. If $S_{red}$ is regular, then $\q_S\in \dmcs_{\wchow=0}$.

\end{theo}
\begin{proof}

I Since $\wchow$ is bounded for any base scheme, in order to prove that a motivic image functor is 
weight-exact it suffices to prove that it preserves Chow motives; see  Proposition \ref{pbw}(\ref{iwfun}). The assertion follows immediately.

II Let $f$ be a smooth morphism. Then we obtain: $f^*(\dmc(Y)_{\wchow=0})\subset \dmc(X)_{\wchow=0}$ by  Proposition \ref{pcisdeg}(\ref{iexch}). Hence $f^*$ is weight-exact  (by the same argument as above). Passing to the limit (using 
Remark \ref{ridmot}(2)) we prove assertion II.3(ii).
 We also obtain that $f^!$ is weight-exact (for a smooth $f$)  using assertion I and   Proposition \ref{pcisdeg}(\ref{ipur}),
i.e.,  we proved assertion II.2.
Besides, the adjunctions yield (by Proposition \ref{pbw}(\ref{iadj})): $f_*$ is right weight-exact, $f_!$ is left weight-exact; i.e., assertion II.1 for $f$ is fulfilled.

Now let $f$ be projective. Then $f_!(\dmc(X)_{\wchow=0})\subset \dmc(Y)_{\wchow=0}$ (since $f_!\circ g_!=(f\circ g)_!$ for any $g$, and $f_!$ commutes with Tate twists). By Proposition \ref{pbw}(\ref{iwfun}) we obtain that $f_!=f_*$ is weight-exact.
Hence if $f$ is a finite universal homeomorphism, $f^*$ is weight-exact also (since it is inverse to $f_*$ by Proposition \ref{pcisdeg}(\ref{itr})). The same argument can be applied to the functor $f^!$ (see Remark \ref{ridmot}(3))  and we obtain assertion II.3(i).
 Next, using the adjunctions and  Proposition \ref{pbw}(\ref{iadj}) again, we obtain that $f^!$ is right weight-exact and $f^*$
 is left weight-exact for an arbitrary projective $f$. So, assertion II.1 is fulfilled also in the case where $f$ is projective. 

Assertion II.1 in the general case follows since any 
smoothly embeddable  morphism (by definition) is  a composition of a closed (i.e., projective) immersion with a smooth 
morphism.


III Since $i_*\cong i_!$ in this case, $i_*$ is weight-exact by assertion II.1. $j^*$ is weight-exact by assertion II.2.

1. $\dmc(U)$ is the localization of $\dmc(X)$ by $i_*(\dmc(Z))$
by  Proposition \ref{pcisdeg}(\ref{igluc}). Hence   Proposition \ref{pbw}(\ref{iloc}) yields the result (see Remark \ref{rlift}).

2.  Proposition \ref{pcisdeg}(\ref{igluc}) 
yields: $\wchow(X)$ is exactly the weight structure obtained by 'gluing $\wchow(Z)$ with $\wchow(U)$' via   Proposition \ref{pbw}(\ref{igluws}) (here we use Proposition \ref{pcisdeg}(\ref{igluwsn})). Hence loc. cit. yields the result (note that $j^*=j^!$). 


IV The assertion can be easily proved by induction on the number of strata  using assertion III.2.

V Let $S_{red}$ be a regular scheme; denote by $v$ the canonical immersion $S_{red}\to S$.
Then $v_*(\q_{\sred})\in \dmcs_{\wchow=0}$ by the definition of $\wchow$. Now, $v^*$ is an equivalence of categories (by    Proposition \ref{pcisdeg}(\ref{itr})) that sends $\q_S$ to $\q_{\sred}$ (see Proposition \ref{pcisdeg}(\ref{iupstar})). Hence (applying the adjunction) we obtain $v_*(\q_{\sred})\cong \q_S$.  So, we proved assertion V.2.

In order to verify assertion V.1 we choose a stratification $S=\cup S_\al$ such that 
all $S^\al_{l,\,red}$ are regular. Since we have $j_l^*(\q_S)=\q_{S_l}\in \dmc(S_l)_{\wchow\ge 0}$ (by assertion V.2), assertion IV implies the result.

\end{proof}

\begin{rema}\label{rmotresing}

1. Assertion III.1 yields that any $C\in \obj\,\chow(U)$ is a retract of some $C'$ coming from $\chow(X)$. This fact can be easily deduced from Hironaka's resolution of singularities (if we believe that the composition of morphisms in $\chow(-)$ could be described in terms of algebraic cycles; cf. Remark \ref{rtwchow}(4)) in the case where $X$ is a variety over a characteristic $0$ field. Indeed, then any  projective regular $U$-scheme $Y_U$ possesses a  projective regular $X$-model $Y$ (since one can resolve the singularities of any projective  model $Y'/X$ of $Y_U$ by a morphism that is an isomorphism over $U$). The author does not know any analogues of this argument in the case of a general (reasonable) $X$ (even with alterations instead of modifications, since it does not seem to be known whether there  exists an alteration of $Y'$ that is \'etale over $U$).

So, assertion III.1 could be called (a certain) {\it motivic resolution of singularities} (over a reasonable $X$). Certainly, applying the assertion repeatedly one can easily extend it to the case where $X\setminus U$ is not necessarily regular (but $U$ is open in $X$).

Alternatively, one can note here that for any $C\in \obj\chow(U)$ we have $j_!(C)\in \cu_{w\le 0}$ and $j_*(C)\in \cu_{w\ge 0}$. Hence the natural morphism $j_!(C)\to j_*(C)$ can be factored through $C_X=j_*(C)_{\wchow \le 0}\in \obj \chow(X)$ (or through $j_!(C)_{\wchow \ge 0}\in \obj \chow(X)$; see Proposition \ref{pbw}(\ref{iext})), whereas $C$ is a retract of $j^*(C_X)$.

Actually, any object of $\chow(U)$  comes from $\chow(X)$ itself; see Theorem 1.7 of \cite{wildic}.

2. The following statement is a trivial consequence of part I of the Theorem along with Proposition 
\ref{pcisdeg}(\ref{ipur}): if $f:X\to S$ is a smooth morphism, then for any $M\in \obj \dmcs,\ m\in \z$ we have: $f^*(M)\in \dmcx_{\wchow\ge m}$ (resp. $f^*(M)\in \dmcx_{\wchow\le m}$) if and only if $f^!(M)\in \dmcx_{\wchow\ge m}$ (resp. $f^!(M)\in \dmcx_{\wchow\le m}$).

\end{rema}

Now we prove that positivity and negativity of objects of $\dmcs$ (with respect to $\wchow$)
can be 'checked at points'; this is a motivic analogue of  \S5.1.8 of \cite{bbd}.

\begin{pr}\label{ppoints}

Let $\sss$  denote the set of 
(Zariski) points of $S$; for a $K\in \sss$ we will denote the corresponding morphism $K\to S$ by $j_K$.

Then $M\in \dmcs_{\wchow\ge 0}$ (resp. $M\in \dmcs_{\wchow\le 0}$) if and only if for any $K\in \sss$ we have $j_K^!(M)\in \dmc(K)_{\wchow\ge 0}$ (resp. $j_K^*(M)\in \dmc(K)_{\wchow\le 0}$); see Remark \ref{ridmot}(2).
\end{pr}
\begin{proof}

By Theorem \ref{tfunctwchow}(II.1) if $M\in \dmcs_{\wchow\ge 0}$ (resp. $M\in \dmcs_{\wchow\le 0}$) then for any immersion $f:X\to S$ we have $f^!(M)\in \dmc(X)_{\wchow\ge 0}$ (resp. $f^*(M)\in \dmc(X)_{\wchow\le 0}$). Passing  to the limits with respect to 
immersions corresponding to points of $S$ (see Remark \ref{ridmot}(2)) yields one of the implications.

We prove the converse implication by Noetherian induction. So, suppose that our assumption is true for motives over any closed subscheme of $S$, and that for some $M\in \obj\dmcs$ we have  $j_K^!(M)\in \dmc(K)_{\wchow\ge 0}$ (resp. $j_K^*(M)\in \dmc(K)_{\wchow\le 0}$) for any  $K\in \sss$.

We should prove that $M\in \dmcs_{\wchow\ge 0}$ (resp. $M\in \dmcs_{\wchow\le 0}$).
By Proposition \ref{pbw}(\ref{iort}) it suffices to verify:  for any $N\in \dmcs_{\wchow\le -1}$ (resp. for any $N\in \dmcs_{\wchow\ge 1}$), and any $h\in \dmcs(N, M)$ (resp. any $h\in \dmcs(M,N)$) we have $h=0$. We fix some $N$ and $h$.

By the 'only if' part of our assertion (that we have already proved)
we have  $j_K^*(N)\in  \dmc(K)_{\wchow\le -1}$ (resp. $j_K^*(N)\in  \dmc(K)_{\wchow\ge 1}$); hence $j_K^*(h)=0$. By Proposition \ref{pcisdeg}(\ref{icont}) we obtain that $j^*(h)=0$ for some open embedding $j:U\to S$, where $K$ is a generic point of $U$.

Now suppose that $h\neq 0$; let $i:Z\to S$ denote the closed embedding that is complementary to $j$. Then Lemma \ref{l4onepo}(II) yields that $\dmcs (i^*(N),i^!(M))\neq \ns$ (resp. $\dmcs (i^*(M),i^!(N))\neq \ns$). Yet $i^*(N)\in \dmc(Z)_{\wchow\le -1}$ (resp. $i^!(N)\in \dmc(Z)_{\wchow\ge 1}$) by Theorem \ref{tfunctwchow}(II), whereas $i^! (M)\in \dmc(Z)_{\wchow\ge 0}$ (resp. $i^* (M)\in \dmc(Z)_{\wchow\le 0}$) by the inductive assumption. The contradiction obtained proves 
our assertion.

\end{proof}

Lastly we prove that 'weights are continuous'.

\begin{lem}\label{lwcont}
Let $K$ be a generic point of  $S$; denote the morphism $K\to S$ by $j_K$.

Let $M$ be an object of $\dmcs$, 
and suppose that 
 $j_K^*M\in \dmck_{\wchow\ge 0}$ (resp.  $j_K^*M\in\dmck_{\wchow\le 0}$). Then there exists an open immersion $j:U\to S$, $K\in U$, such that $j^*M\in \dmc(U)_{\wchow\ge 0}$ (resp. $j^*M\in \dmc(U)_{\wchow\le 0}$).

\end{lem}
\begin{proof}
First consider the case $j_K^*(M)\in \dmck_{\wchow\ge 0}$.
We consider a weight decomposition of $M[1]$: $B\stackrel{g}{\to} M[1]
{\to} A\to B[1]$.
We obtain that $j^*_K(g)=0$ (since $\dmck_{\wchow\le 0}\perp j^*_K(M)[1])$; hence (by Proposition \ref{pcisdeg}(\ref{icont}))  there exists an open immersion $j:U\to S$ ($K\in U$)
such that $j^*(g)=0$. Hence 
$j^*M[1]$ is a retract of $j^*A$. Since $j^*A[-1]\in \dmc(U)_{\wchow\ge 0}$ (see Theorem \ref{tfunctwchow}(II.2)), and $\dmc(U)_{\wchow\ge 0}$ is Karoubi-closed in $\dmc(U)$, we obtain the result.
 
The 
second part of our statement (i.e., the one for the case  $j_K^*(M)\in\dmck_{\wchow\le 0}$) can be easily verified using the dual argument (see Proposition \ref{pbw}(\ref{idual})). 

\end{proof}

\subsection{The 'gluing' construction of $\wchow$ over any  (excellent separated finite-dimensional)  $S$}\label{schowunr}

In this subsection all  
schemes (including the base scheme $S$) will be excellent separated finite-dimensional; we do not assume them to be reasonable. Then we can define the Chow weight structure 
'locally'.
We
explain how to do this (using stratifications and  gluing of weight structures; we call this approach to constructing $\wchow$ the 'gluing method').

First we  describe certain candidates for $\dmcs_{\wchow\ge 0}$ and $\dmcs_{\wchow\le 0}$ (partially they are motivated by Remark \ref{rtwchow}(2)); next we will prove that they yield a weight structure on $\dmcs$ indeed.

For a scheme $X$ we will denote by $\onx$ (resp. $\opx$) the envelope (see the Notation) of
$p_*(\q_P)(s)[i+2s](\cong p_!(\q_P)(s)[i+2s])$ 
in $\dmcx$; here $p:P\to X$ runs through all  
morphisms to $X$ that can be factored as $g\circ h$, where $h:P\to X'$ is a smooth projective morphism, $X'$ is a regular scheme, $g:X'\to X$ is a finite universal homeomorphism,  $s\in \z$, whereas $i\ge 0 $ (resp. $i\le 0$). We denote $\onx\cap \opx$ by $\oz(X)$.

\begin{rema}\label{r1fun}
1. For a morphism $f:Y\to X$ we have (by Proposition \ref{pcisdeg}(\ref{iexch})): $f^*p_!(\q_P)\cong p'_!f'^*(\q_P)=p'_!(\q_{P_Y})$; 
if $p$ is projective (as we demanded in the definition of ($\on(-),\op(-)$)) then we can replace all $-_!$ here with $-_*$. Now suppose that for $X'/X$ as above 
the reduced scheme $Y'_{red}$ associated to $Y'=X'_Y $  is regular. Then $f^*p_*(\q_P)(s)[2s]\in \oz(Y)$.
Indeed,  consider the diagram
\begin{equation}\label{cdred}\begin{CD}
 P_{Y,red}@>{p_r}>> P_Y@>{f_P}>> P\\
 \\@VV{h_{Y,red}}V
@VV{h_Y}V@VV{h}V \\
 Y'_{red}@>{y'_r}>> Y' @>{f'}>> X'\\
 \\@.
@VV{g_Y}V@VV{g}V \\
  @. Y @>{f}>> X
 \end{CD}\end{equation}
where  
$p_r$ and $y'_r$ are the corresponding nil-immersions.
 We have $f^*p_*(\q_P)(s)[2s]
\cong p_{Y*}\q_{P_Y}(s)[2s]\cong p_{Y*}p_{r*}
\q_{P_{Y'_{red}}}(s)[2s]$ (see Remark \ref{ridmot}(3)). We can transform this further into
$(g_Y\circ y'_r)_*h_{Y,red_*}\q_{P_{Y_{red}}}(s)[2s]
\in \oz(Y)$.

Moreover, part \ref{ipura} of Proposition \ref{pcisdeg} (that represents $i^!\q_B$ as a Tate twist of $\q_A$ for an immersion $i:A\to B$ of connected regular schemes)
 easily yields that $f^!p_*(\q_P)(s)[2s]\in \oz(Y)$
if $f$ induces an immersion $Y'_{red}\to X'$ of regular schemes. 
Indeed, we can assume that $Y$ is connected; hence $Y'_{red}$ and $P_{Y,red}$ are connected also. Denote the codimension of $Y'_{red}$ in $ X'$
by $c$; then $p_r\circ f_P:P_{Y,red}\to P$ is an immersion of regular schemes of codimension $c$. Then (arguing as above) we obtain:  $f^!p_*(\q_P)(s)[2s]\cong (g_Y\circ y'_r)_*h_{Y,red_*}(p_r\circ f_P)^!\q_{P}(s)[2s]$. 
Using loc. cit., we transform this into $(g_Y\circ y'_r)_*h_{Y,red_*}\q_{P_{Y_{red}}}(s-c)[2s-2c]
\in \oz(Y)$.




2.  Any morphism $p:X\to S$ for a regular $X$ can be factored through  the underlying reduced subscheme $S_{red}$ of $S$. So, all our descriptions of $\wchow(S)$ (including the ones given below) 
'depend' only on $S_{red}$. This is (certainly) coherent with the (weight-exact) isomorphism $\dmcs\to \dmc(S_{red})$ given by Proposition \ref{pcisdeg}(\ref{itr}). 

Besides, in all the statements of this section  the reader may assume that 
$\onx=\opx=\ns$ unless $X_{red}$ is regular.

3. Moreover, if 
 $X$ is generically 
 of characteristic $0$ 
 then it 
 suffices for our purposes to take $X'=X_{red}$. Indeed, in this case  each finite universal homeomorphism  with
 regular domain is generically of the form $X_{red}\to X$; so we can apply the argument used in the proof of Proposition \ref{pwchowa}(below) for this alternative version of the definition of ($\on(-),\op(-)$) also.
  Hence in the case where $S$ is a  reduced $\q$-scheme, one can assume that all 
the morphisms $p$ (that we use for the description of $\wchow(S)$ given below) 
are smooth projective; this is also true if $S$  is the spectrum of a subring of a number field.

\end{rema}

For a stratification $\al:S=\cup S_l^\al$ we denote by $\onal$ the class $\{M\in \obj \dmcs:\ j_l^! (M)\in \onsl, 1\le l\le n\} $;
$\opal=\{M\in \obj \dmcs:\ j_l^* (M)\in \opsl, 1\le l\le n\} $.

We define: $\dmcs_{\wchow\ge 0}=\cup_\al \onal$, $\dmcs_{\wchow\le 0}=\cup_\al \opal$; here $\al$ runs through all 
stratifications of $S$.

\begin{rema}\label{rnew}
1. It seems that the unions in the definition of $(\dmcs_{\wchow\ge 0},\dmcs_{\wchow\le 0})$ given above are not filtering (if $S$ is not a $\spe\q$-scheme). In particular, we don't have $\on(\al)\subset \on(\al')$ (and $\op(\al)\subset \op(\al')$) for a general subdivision $\al'$ of a stratification $\al$. In order to overcome this difficulty we prove a certain weaker statement instead (see 
 Lemma \ref{lglustr}(3)); it is sufficient for our purposes. In the proof of this result we also treat the question when 
 an element of $\on(\al)$ (or of $\op(\al)$) belongs to   $\on(\al')$ (or to $\op(\al')$, respectively), where $\al'$ is a subdivision of $\al$. 

2. Though we define $(\dmcs_{\wchow\ge 0}, \dmcs_{\wchow\le 0})$ in terms of $\onal$ and $\opal$, there seems to be no way to express our  $\onal$ and $\opal$ in terms of $\wchow$. So, we only use 
$\on(-)$ and $\op(-)$ as  technical notions in the definition of $\wchow$; we could have chosen certain alternative versions of the former (see Remark \ref{rflex}(1) below). Taking all of this into account, the reader should not be scared of 'bad' properties of $\onal$ and $\opal$.

\end{rema}

\begin{lem}\label{lglustr}

1. Let $\de$ be a 
stratification of $S$; we denote the corresponding immersions $S_l^\de\to S$ by $j_l$.
Let $M$ be an object of $\dmcs$.

 Suppose that $j_l^!(M)\in \dmc(S_l^\de)_{\wchow\ge 0}$  (resp. $j_l^*(M)\in \dmc(S_l^\de)_{\wchow\le 0}$) for all $l$.

 Then $M\in \dmcs_{\wchow\ge 0}$ (resp. $M\in \dmcs_{\wchow\le 0}$).

2.   $j_*(\dmc(V)_{\wchow\ge 0})\subset \dmcs_{\wchow\ge 0}$  and  $j_!(\dmc(V)_{\wchow\le 0})\subset \dmcs_{\wchow\le 0}$  for any immersion $j:V\to S$.

3. For any $M\in \dmcs_{\wchow\le 0}$ and $N\in \dmcs_{\wchow\ge 1}(=\dmcs_{\wchow\ge 0}[1])$ there exists a stratification $\al$ of $S$ such that $M\in \opal$, $N\in \onal[1]$.

\end{lem}
\begin{proof}
1. We use induction on the number of strata in $\de$. The $2$-functoriality of motivic upper image functors yields: it suffices to prove the statement for $\de$ consisting of two strata.

So, let $S=U\cup Z$, $Z$ and $U$ are disjoint, $U\neq \ns$ is open  in $S$; we denote the immersions $U\to S$ and $Z\to S$ by $j$ and $i$, respectively.

By the assumptions on $M$, there  exist 
stratifications $\be$ of $Z$ and $\gam$ of $U$ such that $i^!(M)\in \on(\be)$ and $j^!(M)\in \on(\gam)$ (resp. $i^*(M)\in \op(\be)$ and $j^*(M)\in \op(\gam)$). 

We 'unify' $\be$ with $\gam$ and denote the 
stratification of $S$ obtained by $\al$ (for $\# \gam=\Gamma$ we put $S_l^\al=U_l^\gam$ if $1\le l\le \Gamma$ and $S_l^\al=\be_{l-\Gamma}^\gam$ if $l> \Gamma$; note that we really obtain a stratification in our weak sense of this notion this way; see the Notation).
Then  
 the $2$-functoriality of $-^!$ (resp. of $-^*$) yields that
$M\in \onal$ (resp. $M\in \opal$).

2. We choose a   stratification $\de$ containing $V$ (as one of the strata). So we assume that $V=S_v^\de$ for some index $v$. Then it can be easily seen that $j_l^!j_{v*}=0=j_l^*j_{v!}$  for any $l\neq v$ and  $j_v^!j_{v*}\cong 1_{\dm(V)} 
 \cong j_v^*j_{v!}$ (see    Proposition \ref{pcisdeg}(\ref{iglu})). Hence the result follows from
assertion 1.

3. By Remark \ref{r1fun}(1) it suffices to verify: if $\beta, \gamma$ are stratifications of $S$, and  $S_{il}\to S_l^{\beta}$, $S'_{il}\to S_l^{\gamma}$ are  (finite) sets of  
finite universal homomorphisms, 
then there exists  
a common subdivision $\al$ of $\beta,\gamma$ such that all the (reduced) schemes $(S_{il}\times_S {S_{m}^{\al}})_{red}, (S'_{il}\times_S {S_{m}^{\al}})_{red}$ 
are regular. To this end it obviously suffices to prove: if $f:Z\to S$ is an immersion, $g_i:T_i\to Z$ are some finite universal homeomorphisms, 
then there exists a stratification $\de$ of $Z$ such that the schemes
$T_{il}=(T_i\times_Z Z_l^{\de})_{red}$ are regular for all $i$ and $l$.

We prove the latter statement by easy Noetherian induction. Suppose that it is fulfilled for any proper closed subscheme $Z'$ of $Z$. Since all $(T_i)_{red}$ are generically regular, we can choose a (sufficiently small) open non-empty subscheme $Z_1$ of $Z$ such that all of  $(T_{i}\times_Z Z_1)_{red}$ are regular. 

Next, apply the inductive assumption to the scheme $Z'=Z\setminus Z_1$ and the morphisms $g_i'=g_i\times _Z Z'$;
we choose some  stratification $\al'$ of $Z'$ such that all $T'_{il}=(T_i\times_Z {Z'}_l^{\al'})_{red}$ are regular. Then it remains to 'unify' $Z_1$ with
 $\al'$, i.e., we consider the following stratification $\al$: $Z_1^{\al}=Z_1$, and $Z_l^{\al}={Z'}_{l-1}^{\al'}$ for all $l>1$.

\end{proof}

\begin{pr}\label{pwchowa}

I.1. The couple ($\dmcs_{\wchow\ge 0}$, $\dmcs_{\wchow\le 0}$) yields a bounded weight structure $\wchow$ for $\dmcs$.  

2. $\dmcs_{\wchow\ge 0}$ (resp. $\dmcs_{\wchow\le 0}$) is the envelope of $p_*(\q_P)(s)[2s+i]$ (resp. of $p_!(\q_P)(s)[2s-i]$) for $s\in\z$, $i\ge 0$, and $p:P\to S$ being the composition of a smooth projective morphism with a 
finite universal homeomorphism whose base is regular and with an immersion.

II $w(S)$ can be extended to a weight structure $\wchowb$ for the whole $\dms$.

\end{pr}
\begin{proof} 

I We prove the statement by Noetherian induction. So, we suppose that assertions I.1 and I.2 are fulfilled for all proper closed subschemes of $S$. We prove them for $S$.

We denote the envelopes mentioned in assertion I.2 by $(\dmcs_{\wchow'\ge 0},\dmcs_{\wchow'\le 0})$. We should prove that $\wchow$ and $\wchow'$ yield coinciding weight structures for $\dmcs$.

Obviously,  $\dmcs_{\wchow\le 0},\dmcs_{\wchow\ge 0}$, $\dmcs_{\wchow'\le 0}$, and $\dmcs_{\wchow'\ge 0}$ are Karoubi-closed in $\dmcs$, and are semi-invariant with respect to translations (in the appropriate sense).

Now, Lemma \ref{lglustr}(2) yields that $\dmcs_{\wchow'\le 0}\subset \dmcs_{\wchow\le 0}$ and $\dmcs_{\wchow'\ge 0}\subset \dmcs_{\wchow\ge 0}$. Hence in order to verify that $\wchow$ and $\wchow'$ are weight structures indeed, it suffices to verify: 

(i) the orthogonality axiom for $\wchow$ 

(ii) any  $M\in \obj\dmcs$ possesses a weight decomposition with respect to $\wchow'$.

Hence these statements along with the boundedness of $\wchow$ imply assertion I.1.
Besides, Proposition \ref{pbw}(\ref{iuni}) yields that these two statements imply assertion I.2 also, whereas in order to prove I.1 it suffices to verify the boundedness of $\wchow'$ (instead of that of $\wchow$).

Now we verify (i). For some fixed $M\in \dmcs_{\wchow\le 0}$ and $N\in \dmcs_{\wchow\ge 1}$ we check that $M\perp N$. 
By Lemma \ref{lglustr}(3), we can assume that $M\in \opal$, $N\in \onal[1]$ for some stratification $\al$ of $S$.
Hence it suffices to prove  that $\opal\perp \onal[1]$ for 
 any 
  $\al$.
    The latter statement 
    is an easy consequence of    Lemma \ref{l4onepo} (parts I.1 and II).

Now we verify (ii) along with the boundedness of $\wchow'$.
 We choose some generic point $K$ of $S$, 
denote by $K^p$ its perfect closure, and by $j_{K^p}:K^p\to S$ the corresponding morphism.  
We fix some $M$. Since $K^p$ is a reasonable scheme, we have $j_{K^p}^*(M)\in \lan \chow(K^p)\ra$ (see 
Proposition \ref{pcisdeg}(\ref{igenc})). 
Moreover, since $K^p$ is perfect, there exist
some smooth projective varieties $P_{i}/K^p,\
1\le i\le n$,
(we denote the corresponding morphisms $P_{i}\to K^p$ by $p_i$) and some $s\in \z$ such that $j_{K^p}^*(M)$ belongs to the
triangulated subcategory of $\dmc(K^p)$ generated by  $\{p_{i*}(\q_{P_{i}})(s)[2s]\}$.
Now we choose some  finite  
universal homeomorphism $K'\to K$ (i.e., a morphism of spectra of fields corresponding to a finite  purely inseparable
extension) such that $P_{i}$ are defined (and are smooth projective) over $K'$. 
By Proposition \ref{pcisdeg}(\ref{icont},\ref{itr}), 
for the corresponding morphisms $j_{K'}:K'\to S$, $p_i':P_{K',i}\to K'$ 
 we have: $j_{K'}^*(M)$ belongs to the
triangulated subcategory of $\dmc(K')$ generated by  $\{p'_{i*}(\q_{P_{K',i}})(s)[2s]\}$. 
Applying Zariski's main theorem in the form of Grothendieck, we 
can choose a finite universal homeomorphism $g$ from a regular scheme $U'$ whose generic fibre is $K'$
to an open $U\subset S$ ($j:U\to S$ will denote the corresponding immersion) and  smooth projective $h_i:P_{U',i}\to U'$ such that the fibres of $P_{U',i}$ over $K'$ are isomorphic to $P_{K',i}$. Moreover, by Proposition \ref{pcisdeg}(\ref{icont}) we can also assume that $(j\circ g)^*(M)$ belongs to the
triangulated subcategory of $\dmc(U')$ generated by  $\{h_{i*}(\q_{P_{U',i}})(s)[2s]\}$.
Then 
Remark \ref{ridmot}(3)
 yields that $j^*(M)$ belongs to the
triangulated subcategory $D$ of $\dmc(U)$ generated by  $\{(g\circ h_i)_{*}(\q_{P_{U',i}})(s)[2s]\}$.

Since $\id_U$ yields a  stratification of $U$,  the set $\{(g\circ h_i)_{*}(\q_{P_{U',i}})(s)[2s]\}$  is negative in $\dmc(U)$ (since $\opal\perp \onal[1]$ for any  $\al$, as we have just proved). Therefore (by Proposition \ref{pbw}(\ref{igen}--\ref{iwgen})) there exists a weight structure $d$ for $D$ such that $D_{d\ge 0}$ (resp. $D_{d\le 0}$) is the envelope of $\cup_{n\ge 0} \{(g\circ h_i)_{*}(\q_{P_{U',i}})(s)[2s+n] \}$ (resp. of $\cup_{n\ge 0} \{(g\circ h_i)_{*}(\q_{P_{U',i}})(s)[2s-n]\}$). We also obtain that $D_{d\ge 0}\subset \dmc(U)_{\wchow'\ge 0}$ and $D_{d\le 0}\subset \dmc(U)_{\wchow'\le 0}$. 

We denote $S\setminus U$ by $Z$ ($Z$ could be empty); $i:Z\to S$ is the corresponding closed immersion.
By the inductive assumption, 
$\wchow$ and $\wchow'$ yield coinciding  bounded 
 weight structures for $\dmc(Z)$. 

We have the gluing datum $\dmc(Z)\stackrel{i_*}{\to}\dmcs \stackrel{j^*}{\to} \dmc(U)$.
We can 'restrict it' to a gluing datum
$$ \dmc(Z)\stackrel{i_*}{\to} j^*{\ob} (D)\stackrel{j_0^*}{\to} D $$ 
 (see  Proposition \ref{pbw}(\ref{igluws})), whereas $M\in \obj (j^*{\ob}(D))$; here $j_0^*$ is the corresponding restriction of $j^*$. 
  Hence by loc. cit. there exists a weight structure $w'$ for $j^*{\ob}(D)$ such that $i_*$ and $j_0^*$ are weight-exact (with respect to the weight structures mentioned). Hence there exists a weight decomposition 
  $B\to M\to A$ of $M$ with respect to $w'$. Besides, there exist $m,n\in \z$ such that $j^*_0(M)\in \dmc(U)_{\wchow'\ge m}$, $j^*_0(M)\in \dmc(U)_{\wchow'\le n}$, $i^!(M)\in \dmc(Z)_{\wchow'\ge m}$, and $i^*(M)\in \dmc(Z)_{\wchow'\le n}$.
 Hence $A[-1],M[-m]\in \dmcs_{\wchow'\ge 0}$; $B, M[-n]\in \dmcs_{\wchow'\le 0}$; 
 here we apply Proposition \ref{pbw}(\ref{igluwsc}). So, we verified (ii) and the boundedness of $\wchow'$. As was shown above, this finishes the proof of assertion I.

II: immediate from assertion I.1; cf. the proof of Theorem \ref{twchow}.

\end{proof}

\begin{pr}\label{pfunctwchowa}
For the version of $\wchow$ constructed in this subsection, the analogues of  all parts of Theorem \ref{tfunctwchow}, as well as of Proposition \ref{ppoints} and Lemma \ref{lwcont} are fulfilled.

\end{pr}
\begin{proof}
The proof of 
Theorem \ref{tfunctwchow}(I) carries over to our situation without changes.
The same is true for Theorem \ref{tfunctwchow}(II.1--II.2) for the case of a smooth 
$f$. Lemma \ref{lglustr}(2) yields assertion II.1 of Theorem \ref{tfunctwchow} for the case where $f$ is an immersion. The general case 
of loc. cit. follows from these two immediately.

The (analogues of) 
the remaining parts of the Theorem 
follow from (the analogue of) part II via the same arguments as 
in \S\ref{sfwchow}.

\end{proof}

\begin{coro}\label{cpoints}
1. We have $\chows\subset \hwchow(S)$ (see the definition of $\chows$ in \S\ref{stmain}).

2. For a reasonable $S$ the 'alternative' version of $\wchow$ (constructed above) coincides with the version given by Theorem \ref{twchow}(I).

\end{coro}
\begin{proof}
1. It suffices to verify that $p_*(\q_P)\in \hwchow(S)$ for any regular $P$ and a projective morphism $p:P\to S$. By the previous proposition, we obtain $\q_P\in \dmc(P)_{\wchow=0}$; since $p_!\cong p_*$, we obtain the result.

2. Indeed, denote the 'old' version of $\wchow$ by $v$, and the 'alternative' one by $w$. The previous assertion along with Proposition \ref{pbw}(\ref{iwfun}) yields that $1_{\dmcs}$ is weight-exact with respect to $v$ and $w$. Hence 
Proposition \ref{pbw}(\ref{iuni})  yields the result.
\end{proof}

\begin{rema}\label{rflex} 

1. Proposition \ref{pwchowa} also easily yields that we could have considered larger 
$\oz(X)$ (see the beginning of the subsection): one can take  $\oz(X)=\chow(X)$. 



2. Thus the results of this subsection (in particular) yield a collection of new descriptions of $\wchow(S)$ for the case of a reasonable $S$ (cf. also Remark \ref{r1fun}(3)).

3.  
In the first draft of this paper (only) the gluing method of constructing $\wchow$ was used (this approach was first
proposed in Remark 8.2.4(3) of \cite{bws}, that was in its turn inspired by \cite{bbd}). Next the author proved part 1 of the Corollary. Then (in order to deduce our main results) it remained to note 
that $\chows$ generates $\dmcs$. Luckily, it was easy to prove the negativity of $\chows$ (without relying on the gluing construction of $\wchow$; see Lemma \ref{l4onepo}(I.1)); so the proof was simplified (for a reasonable $S$; note still that the scheme of the proof of loc. cit. is similar to 
the chain of arguments that yields the first part of the Corollary). Yet (as we have noted just above) even for a reasonable $S$ the gluing method gives us some 'new' descriptions of $\wchow$. 
The main disadvantage of the gluing method is that it does not yield an explicit description of the whole $\dmcs_{\wchow=0}$ (though we can describe it as the intersection of
$\dmcs_{\wchow\le 0}$ with $\dmcs_{\wchow\ge 0}$).

4. Let $K$ be a generic point of $S$, $j_K:K\to S$ is the corresponding morphism; assume that for $M\in \obj \dmcs$
we have $j_K^*(M)\in \dmc(K)_{\wchow=0}$. By Theorem \ref{tfunctwchow}(II.3(i)), for any finite universal homeomorphism $K'
\to K$, $j_{K'}:K'\to S$ being the corresponding composition, we also have $j_{K'}^*(M)\in \dmc(K')_{\wchow=0}$. Hence an argument
similar to that  used in the proof of Proposition \ref{pwchowa} easily yields: there exists an open immersion $j:U\to S$, $K\in U$, such that $j^*(M)$ is a retract of 
$(g\circ h)_*\q_P(s)[2s]$, where $h:P\to U'$ is a smooth projective morphism, $U'$ is a regular scheme, $g:U'\to U$ is a finite universal homeomorphism,  $s\in \z$.

5. In \cite{bmm} the author
reduces the conjecture on the existence of the  motivic $t$-structure on $\dmcs$ to the case where $S$ is a universal domain. 
To this end certain gluing arguments are very useful. 

6. Possibly one can use the methods of \cite{hebpo} in order to extend the weight-exactness results for the motivic image functors given by Proposition \ref{pwchowa} to not necessarily 
smoothly embeddable morphisms.

7. Motives with $\z$-coefficients are more 'mysterious' than those with $\q$-ones; yet the author has constructed a certain version of the Chow weight structure for them in \cite{brmz}.

\end{rema}

\section{Applications to (co)homology of motives and other matters}\label{sapcoh}

In \S\ref{swc}  we study weight complexes for $S$-motives (and their compatibility with weight-exact motivic image functors).

In \S\ref{skth} we prove that  $K_0(\dmcs)\cong K_0(\chows)$ (following \cite{bws}), and define a certain 'motivic Euler characteristic' for (separated finite type) $S$-schemes.

In \S\ref{schws} we consider Chow-weight spectral sequences and filtrations for (co)homology of $S$-motives (following \S2.4 of \cite{bws}). We observe that Chow-weight filtrations yield Beilinson's 'integral part' of motivic cohomology (see \S2.4.2 of \cite{bei85} and \cite{scholl}).

In \S\ref{smshgen} we verify that Chow-weight spectral sequences  
 yield the existence of weight filtrations for the 'perverse \'etale homology' of motives over finite type $\q$-schemes (this is not at all automatic for mixed perverse sheaves in characteristic $0$).

In \S\ref{srws} we introduce the notion of  relative weight structure. The axiomatics of those was chosen to be an abstract analogue  of Proposition 5.1.15 of \cite{bbd}. Several properties of relative weight structures are parallel to those for 'ordinary' weight structures.

In \S\ref{sperv} we study the case where $S=X_0$ is a variety over a finite field. In this case the category $\dbm(X_0,\ql)$ of mixed complexes of sheaves possesses a relative weight structure whose heart is the class of pure complexes of sheaves.
Since the \'etale realization of motives preserves weights, we obtain that  (Chow)-weight filtrations for some cohomology theories can be described in terms of the category $\dbm(X_0,\ql)$.

In this section we will always assume that  our base schemes are reasonable. Yet  we also could  have used the 'gluing' version of $\wchow$ (and consider any finite-dimensional noetherian $S$; the main difference is that we would have to put $\hwchow$ instead of $\chow(-)$ everywhere).

\subsection{The weight complex functor for $\dmcs$} \label{swc}

We prove that the weight complex functor (whose 'first ancestor'
was defined by Gillet and Soul\'e)
can be defined for $\dmcs$.

\begin{pr}\label{pwc}
1. The embedding $\chows\to K^b(\chows)$ factors 
 through a certain exact conservative {\it weight complex} functor
$t_S:\dmcs\to K^b(\chows)$.

2. For $M\in \obj \dmcs$, $i,j\in \z$, we have $M\in \dmcs_{[i,j]}$ (see Definition \ref{dwstr}(IV)) if and only if  $t(M)\in K(\chows)_{[i,j]}$ (see Remark \ref{rstws}). 

3. For schemes $X,Y$ let $F:\dmcx\to \dmcy$ be a weight-exact functor of triangulated categories (with respect to the Chow weight structures for these categories; so $F$ could be equal to $i_!$ for a finite type 
projective morphism $i:X\to Y$, or to $j^*$ for a finite type smooth morphism $j:Y\to X$) that possesses a differential graded enhancement. Denote by $F_{K^b(\chow)}$
the corresponding functor $K^b(\chowx)\to K^b(\chowy)$.
Then there exists a choice of $t_X$ and $t_Y$  that makes the diagram
$$\begin{CD}
\dmcx@>{F}>>\dmcy\\
@VV{t_X}V@VV{t_Y}V \\
K^b(\chowx)@>{F_{K^b(\chow)}}>>K^b(\chowy)
\end{CD}$$
commutative up to an isomorphism of functors.

\end{pr}
\begin{proof}
1. By  Proposition 5.3.3 of \cite{bws}, this follows from the existence of a bounded Chow weight structure on $\dmcs$ along with the fact that it admits a differential graded enhancement (see Definition 6.1.2 of ibid.; note also that an alternative 'f-category' version of the weight complex functor  was constructed by Beilinson and Schn\"urer; see \S7 of \cite{schn}). The latter property of $\dms$ can be easily verified since it can be described in terms of the derived category of $qfh$-sheaves over $S$; see Theorem 16.1.2 of \cite{degcis} (and also cf. \S6.1  of \cite{bev}).

2. Immediate from   Theorem 3.3.1(IV) of \cite{bws}.

3. We use the notation and definitions of \S2 of \cite{mymot} (that originate mostly from \cite{bk}).

Since $\dmcx=\lan \chows\ra$, 
we can assume that $\dmcx=\trp(C_X)$, where $C_X$ is a negative triangulated category such that $H(C_X)=\chowx$ (see  Remark 2.7.4(2) of ibid.). Replacing $\dmcy$ by an equivalent category, we may also assume (similarly) that $\dmc(Y)=\trp(C_Y)$, where $C_Y$ is a negative triangulated category such that $H(C_Y)=\chowy$, and $F=\prt(F')$ for some differential graded functor $C_X\to C_Y$. Arguing as in \S6.1 of ibid, we obtain that it suffices to apply $\trp$ to the following diagram:
$$\begin{CD}
C_X@>{F'}>>C_Y\\
@VV{}V@VV{}V \\
H(C_X)@>{H(F')}>> H(C_Y)
\end{CD}$$

\end{proof}

\begin{rema}\label{rwc}

1.  The 'first ancestor' of our weight complex functor  was defined by Gillet and Soul\'e in \cite{gs}. Weight complex for a general triangulated category $\cu$ endowed with a weight structure was defined  in \cite{bws}. Even in the case where $\cu$ does not admit a differential graded enhancement, one can still define a certain 'weak' version of the weight complex; see \S3 of ibid. (and this version does not depend on any choices). It follows that  for $M\in \obj \dmcs$ the isomorphism class of $t_S(M)$ (in $K^b(\chows)$) does not depend on any choices (see ibid.).

2. In \cite{sg} a functor $h$ from the category of Deligne-Mumford stacks
over $S$ (with morphisms being proper morphisms over $S$) to the category of complexes over a certain category of {\it $K_0$-motives}  was constructed; Gillet and Soul\'e considered base schemes satisfying rather restrictive conditions (mostly, of dimension $\le 1$).
We conjecture:  for a 
finite type  
morphism $p:X\to S$ there is a functorial isomorphism $h(X)\to t(\mg_c(X))$, where $\mg_c(X)=p_*p^!(\q_S)$. For $S$ being the spectrum of a characteristic $0$ field this was (essentially) proved in \S6.6 of \cite{mymot}. Note here: though the category of {\it $K_0$-motives} is somewhat 'larger' than $\chows$, it very probably suffices to consider its 'Chow' part (this would be the category $\chows$ considered  in \cite{haco}).

Note that our definition of a weight complex (for $\mg^c(X)$) gives it much more functoriality in $X$ than it was established \cite{sg}; we also study its functoriality with respect to $S$,  and relate it to (co)homology (below).

Besides, we can restrict our definition of weight complexes to (motives with compact support of) quotient stacks (cf. Definition 1.2 of \cite{sg}).
For a finite $G,\ \#G=n$,
acting on a finite type scheme $X/S$ one can take $\mg^c(X/G)=a_{G*}
\mg^c(X)\in \obj\dmcs$. 
Here $a_G$ is the idempotent morphism (correspondence)
$\frac{\sum_{g\in G}g}{n}:X\to X$.  Certainly, for
$G=\{e\}$ we will have $t_S(\mg^c(X/G))=t_S(\mg^c(X))$.

3. Theorem \ref{twchow} along with the results of \cite{bws} also imply:  $t_S$ can be extended to an exact functor $\dms\to K(B\chows)$, where $B\chows$ is the big hull of $\chows$ (see  Definition \ref{dwstr}(IX)). 

4. One can also define exact (and conservative) {\it higher truncation functors} $t_{S,N} $ from $\dmcs$ to certain triangulated $\dmcs_N$ for all $N\ge 0$; cf. \S6.1 of \cite{mymot}.
Here $t_{S,0}=t_S$; $\dmcs_N$ is constructed using differential graded methods that enable  'killing all morphisms' from $\dmcs_{\wchow=0}$ to $\dmcs_{\wchow=i}$ for $i<-N$ and preserving such morphisms for $i\ge -N$. $t_{S,N'}$ factors through $t_{S,N}$ for any $N'>N$, and the $2$-limit of $\dmcs_N$ equals $\dmcs$. $t_{S,N}$ would also satisfy the analogue of Theorem 6.2.1 of ibid. Yet it seems that $t_S=t_{S,0}$ is the most interesting of these 
truncation functors.

\end{rema}

\subsection{$K_0(\dmcs)$ and a motivic Euler characteristic} \label{skth}

Now we  calculate $K_0(\dmcs)$ and study a certain Euler characteristic for (finite type separated) $S$-schemes.

\begin{pr}\label{pkz}

1. We define $K_0(\chows)$  as the Abelian group 
 whose
generators are $[M]$, $M\in \obj \chows$,
and the relations are $[B]=[A]+[C]$ if
$A,B,C\in\obj \chows$ and $B\cong A\bigoplus C$.
For $K_0(\dmcs)$
we take similar generators and  set $[B]=[A]+[C]$ if
$A\to B\to C\to A[1]$ is a distinguished triangle.

Then the embedding $\chows\to \dmcs$  yields
an isomorphism $K_0(\chows)\cong K_0(\dmcs)$.

2. For the correspondence $\chi:X\mapsto [p_*p^!(\q_S)]$ (here $p:X\to S$ is a finite type separated morphism) from the class of finite type separated $S$-schemes to $K_0(\dmcs)\cong K_0(\chows)$ we have: $\chi(X\setminus Z)=\chi(X)-\chi(Z)$ if $Z$ is a closed subscheme of $X$.
\end{pr}
\begin{proof}
1. Immediate from (part I of) Theorem \ref{twchow} and  Proposition 5.3.3(3) of \cite{bws}.

2. Denote the immersion $Z\to X$ by $i$, and the complementary immersion by $j$.
By Proposition \ref{pcisdeg}(\ref{iglu}) for any $M\in \obj\dmcx$
 we have a distinguished triangle
$i_*i^!(M) \to M \to j_*j^!(M)$ (note that $i_!\cong i_*$ and $j^!=j^*$).
Now for $M=p^!(\q_S)$ this triangle specializes to the triangle $i_*(p\circ i)^!(\q_S) \to p^!(\q_S) \to j_*(p\circ j)^!(\q_S)$. It remains to apply $[p_*(-)]$ and the definition of $K_0(\dmcs)$ to obtain the result.

\end{proof}

\begin{rema}
1. Assertion 2 is a vast extension of Corollary 5.13 of \cite{sg}. It allows us to define certain motivic Euler characteristics for (finite type separated) $S$-schemes.

2. We hope that our results would be useful for the theory of motivic integration.

Note in particular: we obtain that any (not necessarily weight-exact!) motivic image functor $\dmcx\to \dmc(Y)$ induces a group homomorphism $K_0(\chow(X))\to K_0(\chow(Y))$.

 Besides, in contrast to the 'classical' case (where $S$ is the spectrum of a field)   there does not seem to exist a 'reasonable' (tensor) product for $\chows$. Yet $\dmcs$  is a tensor triangulated category; hence one can use assertion 1 in order to define a ring structure on $K_0(\chows)$.


\end{rema}

\subsection{Chow-weight spectral sequences and filtrations}\label{schws}

Now we discuss (Chow)-weight spectral sequences and 
 filtrations for homology and cohomology of motives.  We note that any weight structure yields certain weight spectral sequences for any (co)homology theory; the main distinction of the result below from the general case (i.e., from Theorems 2.3.2 and 2.4.2 of ibid.) is that $T(H,M)$ always converges  (since  $\wchow$ is bounded). Since below we will be mostly interested in weight filtrations for cohomological functors,  we will define them in this  situation only; certainly, dualization is absolutely no problem (cf. \S2.1 of ibid.)

\begin{pr}\label{pwss}
Let $\au$ be an abelian category.

I Let $H:\dmcs\to \au$ be a homological functor; for any $r\in \z$ denote $H\circ [r]$ by $H_r$.

For an $M\in \obj\dmcs$ we denote by $(M^i)$ the terms of $t(M)$ (so $M^i\in \obj \chows$; here we can take any possible choice of $t(M)$). 


Then the following statements are valid.

1. There exists a ({\it Chow-weight})  spectral sequence $T=T(H,M)$ with $E_1^{pq}=
H_q(M^p)\implies H_{p+q}(M)$; the differentials for $E_1T(H,M)$ come from $t(M)$.

2. $T(H,M)$ is $\dmcs$-functorial in $M$ (and does not depend on any choices) starting from $E_2$.

II.1. Let $H:\dmcs\to \au$ be any contravariant functor.
Then for any $m\in \z$ the object $(W^{m}H)(M)=\imm (H(\wchow_{\ge m}M)\to H(M))$
does not depend on the choice of $\wchow_{\ge m}M$; it is functorial in $M$.

We call the filtration of $H(M)$ by $(W^{m}H)(M)$ its {\it Chow-weight} filtration.

2. Let $H$ be a cohomological functor.  For any $r\in \z$ denote $H\circ [-r]$ by $H^r$. 

Then the natural dualization of assertion I is valid.
For any $M\in \obj \dmcs$ we have a spectral sequence 
with $E_1^{pq}=
H^{q}(M^{-p})$; it converges to $H^{p+q}(M)$. Moreover, the step of filtration given by ($E_{\infty}^{l,m-l}:$ $l\ge k$)
 on $H^{m}(X)$ equals $(W^k H^{m})(M)$ (for any $k,m\in \z$). $T$ is functorial in $H$ and $M$ starting from $E_2$.

\end{pr}

\begin{proof}

I Immediate from Theorem 2.3.2 of ibid. 

II.1. This is  Proposition 2.1.2(2) of ibid.

2. Immediate from Theorem 2.4.2 of ibid.

\end{proof}

\begin{rema}\label{rintel}

1.  We obtain certain {\it Chow-weight} spectral sequences and
filtrations for any (co)homology of motives. In particular, we have them
for (rational) \'etale and motivic (co)homology of motives.
For the latter theory, we obtain certain results that cannot be proved using 'classical' (i.e., Deligne's) methods, since the latter heavily rely on the degeneration of (an analogue of) $T$ at $E_2$. We will conclude this subsection by studying an example of this sort; we obtain a result that (most probably) could not be guessed using the 'mixed motivic intuition'.

2. $T(H,M)$ can be naturally described in terms of the {virtual $t$-truncations} of $H$
(starting from $E_2$); see \S2.5 of \cite{bws} and \S\ref{svirt} below.

3. We obtain that any (co)homology of any $M\in \obj\dmcs$ possesses a filtration by subfactors of (co)homology of regular projective $S$-schemes.

4. Actually, $(W^{m}H)(M)=\imm (H(\wchow_{\ge m}M)\to H(M))$
does not depend on the choice of $\wchow_{\ge m}M$ and is functorial in $M$ for any contravariant $H:\dmcs\to \au$; see Proposition 2.1.2(2) of \cite{bws}.

\end{rema}

The functoriality of Chow-weight filtrations has quite interesting consequences.

\begin{pr}
Suppose that a scheme $X$ is reasonable; adopt the notation of Proposition \ref{pcisdeg}(\ref{iglu}). 
Let $H:\dmcx\to \au$ be a 
contravariant functor. Then
for all $N\in  \dmc(U)_{\wchow \le 0}$ consider $G(N)= (W^0H)(j_!(N))\subset H(j_!(N))$.

I The following statements are fulfilled.

1. $G(N)$ is $\dmc(U)$-functorial in $N$.

2. (For any $N\in  \dmc(U)_{\wchow \le 0}$) $G(N)$ is a quotient of $H(M)$ for some $M\in \dmcx_{\wchow=0}$.

II Let now $N=j^*(M)(=j^!(M))$ for $M\in \dmcx_{\wchow=0}$. Then the following statements are fulfilled.

1. $G(N)=\imm H(M)\to H(j_!j^!(M))$  (here we apply $H$ to the morphism $j_!j^!(M)\to M$ coming from the adjunction $j_!: \dmc(U) \leftrightarrows \dmcx:j^!$).

2. Let $H=\dmcx(-,\q_X(q)[p])$ for some $q,p\in \z$, and $M=f_*(\q_R)$ for a projective $f:R\to X$, $R$ is regular. Then $G(N)\cong \imm (\operatorname{Gr}^p_\gamma(K_{2q-p}(R)\otimes\q)\to \operatorname{Gr}^p_\gamma(K_{2q-p}(R_U)\otimes\q))$.

\end{pr}
\begin{proof}
I. Assertion 1 is an immediate consequence of Remark \ref{rintel}(4). Next (by Theorem \ref{tfunctwchow}(II.1))  $N\in  \dmcu_{\wchow \le 0}$ implies $j_!(N) \in \dmcx_{\wchow\le 0}$. Hence $w_{\ge 0}
j_!(N) \in \dmcx_{\wchow=0}$ (see Proposition \ref{pbw}(\ref{iext})). This yields assertion 2.

II We have 
$i_*i^*(M) \in  \dmcx_{\wchow\le 0}$. 
Hence the distinguished triangle $$ i_*i^*(M)\to j_!(N)[1](=j_!j^!(M)[1]) \to M[1]$$ (given by Proposition \ref{pcisdeg}(\ref{iglu})) yields a weight decomposition of
$j_!(N)[1]$. This yields assertion II.1.

Now, Proposition \ref{pcisdeg}(\ref{ibormo}) (see also Corollary 14.2.14  of \cite{degcis}) yields $\dmcx(f_*(\q_R),\q_X(q)[p])\cong \operatorname{Gr}^p_\gamma(K_{2q-p}(R)\otimes \q)$.
 Besides, the adjunction $j_!: \dmcu \leftrightarrows \dmcx:j^!$ yields that $$\dmcx(j_!j^!(M),\q_X(q)[p])\cong \dmcu(j^!(M),j^!(\q_X)(q)[p])\cong  \dmcu(f_{U*}(\q_{R_U}),\q_U(q)[p]) \cong \operatorname{Gr}^p_\gamma(K_{2q-p}(R_U)\otimes\q).$$
 It remains to 
 compare the rows of the diagram 
 $$\begin{CD}
H(M)  @>{}>>H(j_!j^!(M))\\
@VV{}V@VV{}V \\
 \operatorname{Gr}^p_\gamma(K_{2q-p}(R)\otimes\q) @>{}>> \operatorname{Gr}^p_\gamma(K_{2q-p}(R_U)\otimes\q)
\end{CD}$$ 

The results of \cite{degcis} (we need the naturality of the isomorphism given by Corollary 14.2.14 of ibid. and also 
 of Proposition \ref{pcisdeg}(\ref{iexch})) yield  that this diagram can be made commutative by (possibly) modifying the left 
column by an isomorphism; this finishes the proof of assertion II.2.

\end{proof}

\begin{rema}
1.  Thus one may say that  $(W^0H)(j_!(N))$
 yields the 'integral part' of  $H(j_!(N))$: we obtain the subobject of  $H^*(j_!(N))$ that  'comes from a nice $X$-model' of $N$ if the latter exists, and a factor of $H(M)$ for some $M\in \dmcx_{\wchow=0}$ in general;
 cf. \cite{bei85}, \cite{scholl}, and  \cite{scholfcohom}. Note here
 that one can also consider $N\in \chow(K)$ for $K$ being a generic point of $X$, since any such $N$ can be lifted to a Chow motif over some $U$ ($K\in U$, $U$ is open in $X$), by Theorem \ref{tfunctwchow}(III.1) combined with Proposition \ref{pcisdeg}(\ref{icont}); cf. also Remark 1.11 of \cite{wildic}. 

 Note that this construction enjoys 'the usual' $\chow(K)$-functoriality; this is an easy consequence of 
 Proposition \ref{pcisdeg}(\ref{icont}) and Remark \ref{rtwchow}(1).

2. Hence we proved  an alternative to Theorem 1.1.6 of \cite{scholl}. 
In more detail: in the setting of our assertion II.2 we can take $X$ being the spectrum of a Dedekind domain, and 
$U$ 'approximating' 
the spectrum $K$ of its fraction field; then the (corresponding) restriction of $H(j_!(N))$ to Chow motives over $K$ satisfies the conditions of loc. cit. Hence in this particular case our version of the 'integral part' of motivic cohomology coincides with the Scholl's one.

An alternative 'mixed motivic' approach to the construction of this integral part is given by Corollary 1.10 of \cite{wildic}. It is equivalent to our one for $H$ being the motivic cohomology (as in assertion II.2); yet the author does not know how to extend loc. cit. to the case of a general $H$.

3. Our description of the 'integral part' of cohomology is very short  and does not rely on any conjectures (in contrast to the description given in \cite{scholfcohom}).

4. It could also be interesting to consider $W^lH^*(j_!(N))$ for $l< 0$.
\end{rema}

\subsection{Application to mixed sheaves: the 'characteristic $0$' case}\label{smshgen}

Suppose that $S'$ is a finite type $\spe \z$-scheme.
Denote by 
$\hetl$ the \'etale realization functor $\dmc(S')\to \dshslz$, where
$\dshslz$ is the Ekedahl's derived category of $\ql$-constructible sheaves (cf. \S1 of \cite{huper}); it could be called \'etale homology (so, it is covariant). 
We will assume below that  $\hetl$  converts the motivic image functors into the corresponding functors for $\dshl(-)$ (it seems that the existence of such a realization is not fully established in the existing literature; yet a forthcoming paper of Cisinski and Deglise should close this gap). 

Now, let $S$ be a finite type (separated) $\q$-scheme. Presenting it as a projective limit of certain $S'^i$ that are flat of finite type over an open subscheme of $\spe \zol$, and using coherent  functoriality of motives and \'etale sheaves, one can construct an exact (covariant) realization functor $\mathbb{H}:\dmcs\to \dsh(S)$, where the latter is the category 
of mixed complexes of $\ql$-\'etale sheaves
as considered    in \cite{huper} (i.e., it is the inductive limit of derived categories of complexes sheaves that are constructible and mixed with respect to 'horizontal' stratifications of 'models' of $S$). Indeed, the arguments of ibid. can be applied here without any problem. The key observation here is that 
$\q_S$ is  mixed (as an object of $\inli D^bSh^{et}(S'^i,\ql)\supset \dsh(S)$), whereas (the \'etale sheaf) image functors preserve mixedness.

The functoriality properties of $\mathbb{H}$ also yields that 
it sends Chow motives over $S$ to pure complexes of sheaves (of weight $0$; see Definition 3.3 of \cite{huper}). Indeed, it suffices to note that $\mathbb{H}$ 
sends $\q_X$ for a regular $X$ to an object of $\dsh(X)$  of weight $0$, whereas
 $f_!=f_*$ for a projective $f$ preserves weights of sheaves (see 
 the Remark succeeding Definition 3.3 of 
 \cite{huper}).
We easily obtain the following statement.

\begin{pr}\label{perz} 
Take $H^{per}$ being the perverse \'etale homology theory,  i.e., 
$H^{per}_i(M)$ (for any $M\in \obj \dmcs$, $i\in\z$) is the $i$-th homology of $\mathbb{H}(M)$ with respect to the perverse $t$-structure of $\dsh$ (see  Proposition 3.2 of \cite{huper}).
 Then  $H^{per}_i(M)$ have weight filtrations (defined using Definition 3.7 of loc. cit., for all $i\in\z$), i.e., it has a  
 increasing filtration whose $j$-th factor is of weight $j$.

\end{pr}
\begin{proof} The existence (and convergence) of the corresponding Chow-weight spectral sequence $T=T(H^{per},M)$ (see Proposition \ref{pwss})(I.1), or Theorem 2.3.2 of \cite{bws}) yields that $H^{per}_i(M)$ possesses a (Chow-weight) filtration whose $r$-th factor (for any $r\in \z$) is a subquotient of $H_{i+r}^{per}(M^{-r})$. Now, $H_{i+r}^{per}(M^{-r})$
is of weight $i+r$ (since this is true for $M^{-r}[r+i]$; cf. Theorem 5.4.1 of \cite{bbd}). 
By Proposition 3.4 of 
ibid. we obtain that the same is true for any its subquotient. 
Shifting the indexing of the Chow-weight filtration, we obtain a filtration of $H^{per}_i(M)$ with the properties required.
 \end{proof}

 \begin{rema}
 Note that the existence of weight filtrations for mixed perverse sheaves  is not at all automatic (in this setting); see 
 the Warning preceding Proposition 3.4 of 
 loc. cit.
 \end{rema}


\subsection{Relative weight structures}\label{srws}

In order to define weights for  mixed complexes of sheaves (over a finite field), we have to generalize the definition of a weight structure.

\begin{defi}\label{drwstr}

I Let $F:\cu\to\du$ be an exact 
 functor (of triangulated categories).

A pair of extension-stable Karoubi-closed 
subclasses $\cu_{w\le 0},\cu_{w\ge 0}\subset\obj \cu$ for
a triangulated category $\cu$ will be said to define a relative weight
structure $w$ for $\cu$ with respect to $F$ (or just an $F$-weight structure)
if
they  satisfy the following conditions.

(i) {\bf 'Semi-invariance' with respect to translations}.

$\cu_{w\le 0}\subset \cu_{w\le 0}[1]$, $\cu_{w\ge 0}[1]\subset
\cu_{w\ge 0}$.

(ii) {\bf Weak orthogonality}.

$\cu_{w\le 0}\perp \cu_{w\ge 0}[2]$.

(iii) {\bf $F$-orthogonality}.

$F$ kills all morphisms between
$\cu_{w\le 0}$ and $\cu_{w\ge 0}[1]$.

(iv) {\bf Weight decompositions}.

 For any $M\in\obj \cu$ there
exists a distinguished triangle
\begin{equation}\label{rwd}
B\to M\to A\stackrel{f}{\to} B[1]
\end{equation} such that $B\in \cu_{w\le 0}, A\in \cu_{w\ge 0}[1]$.

II We define $\cu_{w\ge i}$, $\cu_{w\le i}$,
$\cu_{w= i}$, $\cu_{[i,j]}$, bounded relative weight structures, and
$\cu^b$ similarly to Definition \ref{dwstr}.

We will call the class $\cu_{w= 0}$ the heart of $w$ (we will not define the category $\hw$).

We will use the same notation for weight truncations with respect to $w$ as the one introduced in Remark \ref{rstws}. We define 
weight-exact functors for relative weight structures as in Definition \ref{dwstr}(VI) (i.e., we do not mention the corresponding $F$'s in the definition).

III Let $H$ be a 
full subcategory of a triangulated $\cu$.

We will say that $H$ is {\it $F$-negative} if
 $\obj H\perp (\cup_{i>1}\obj (H[i]))$ and $F$ kills all morphisms between $H$ and $H[1]$.

\end{defi}

\begin{rema}
1. Any weight structure on $\cu$ is a relative weight structure with respect to $F=1_{\cu}$.

2. An $F$-weight structure is also a  $G\circ F$-weight structure for any exact functor $G:\du\to \eu$ (for any triangulated $\eu$). In particular, one can always take $F=0$. 
 Hence we do not lose  generality by adding the $F$-orthogonality axiom to the definition of relative weight structures.
 Besides, those properties of relative weight structures that do not depend on the choice of $F$ are certainly valid without this axiom.
 The main reason to put the $F$-orthogonality axiom together with the weak orthogonality one is that these conditions could be tracked down using similar methods.

 3. The weak orthogonality axiom is 
 is a strengthening the higher Hom decomposition condition that was studied in Appendix B of \cite{posat}. In particular, our Proposition \ref{pbrw}(\ref{irgen}) can be (more or less) easily deduced from the results of ibid.


\end{rema}

Now we will extend to relative weight structures several properties of weight structures. We will skip those parts of the   proofs that do not differ much from the ones in \cite{bws} (for 'usual' weight structures); we will concentrate on the distinctions.

\begin{pr} \label{pbrw}

Let $F:\cu\to\du$ be an exact functor (of triangulated categories).

In all assertions 
except (\ref{irgen})
we will also assume that $w$ is  a relative weight structure for $\cu$ with respect to $F$.

\begin{enumerate}

\item\label{irdual}
$(C_1,C_2)$ ($C_1,C_2\subset \obj \cu$) define an $F$-weight structure on $\cu$  if and only if
$(C_2^{op}, C_1^{op})$ define a relative weight structure on $\cu^{op}$ with respect to $F^{op}$; here $F^{op}:\cu^{op}\to \du^{op}$ is the functor obtained from $F$ by inverting all arrows.

\item\label{irext} All $\cu_{[i,j]}$ are extension-stable.

\item\label{irffunctwd}  Let $l\ge m\in \z$, $M,M'\in
\obj \cu$. Let
 weight decompositions
        of $M[-m]$   and $M'[-l]$  be fixed; we consider the corresponding triangles $w_{\le m}M\stackrel{b}{\to} M\stackrel{a}{\to} w_{\ge m+1}M  $ and $w_{\le l}M'\stackrel{b'}{\to} M'\stackrel{a'}{\to} w_{\ge l+1}M'$.

         Then
for any $g\in \cu(M,M')$ there exists some morphism of  distinguished triangles
\begin{equation}\begin{CD} \label{d2x3r1}
F(w_{\le m}M)@>{F(b)}>>F(M) @>{F(a)}>> F(w_{\ge m+1}M)
\\
@VV{}V@VV{F(g)}V @VV{}V\\
F(w_{\le l}M')@>{F(b')}>>F(M') @>{F(a')}>> F(w_{\ge l+1}M')
\end{CD}\end{equation}

\item\label{irsffunctwd} In addition to the assumptions of the previous assertion, suppose that $l>m$.

Then there also exists a morphism of triangles 

\begin{equation}\begin{CD} \label{d2x3r2}
w_{\le m}M@>{b}>>M @>{a}>> w_{\ge m+1}M @>{f}>>w_{\le m}M[1]
\\
@VV{c}V@VV{g}V @VV{d}V  @VV{c[1]}V\\
w_{\le l}M'@>{b'}>>M' @>{a'}>> w_{\ge l+1}M'@>{f'}>>w_{\le l}M'[1]
\end{CD}\end{equation}
Moreover,  $(g,a,a')$ determine $F(d)$ uniquely;
$(g,b,b')$ determine $F(c)$ uniquely.

\item\label{iwpt}
 For any $M\in \obj \cu$ any
 choices of $w_{\ge i}M$ (and of the arrows $a_i: M\to w_{\ge i}M$ for all $i\in \z$)
 can be completed to a  {\it weight Postnikov tower} for $M$ (cf. Definition 1.5.8 of \cite{bws}), i.e., for all $j\in \z$ we can choose some
morphisms $c_j:w_{\ge j}M\to w_{\ge j+1}M$ that are compatible with $a_j,\,a_{j+1}$, and for any choice of these $c_j$ we have:  $M^j=\co (c_{-j}(M))[j-1]\in \cu_{w=0}$.

\item\label{ibwpt}
For $j_0,j_1\in \z$, we can choose a weight Postnikov tower for $M$ such that $w_{\ge j}M=0$ for $j>j_1$ and  $=M$ for $j\le j_0$  if and only if $M\in \cu_{[j_0,j_1]}$.

We will call such a  weight Postnikov tower a
{\it bounded} one.

\item\label{irwe} Let $w$ be bounded, $G$ be an exact functor $\cu\to\cu'$; suppose that $\cu'$ is endowed with a relative weight structure (with respect to some exact functor $F':\cu'\to \du'$).

Then $G$ is left (resp. right) weight-exact if and only if $G(\cu_{w=0})\subset \cu'_{w'\le 0}$ (resp. $G(\cu_{w=0})\subset \cu'_{w'\ge 0}$).

\item \label{irgen} 

Let  $H\subset  \cu$ be $F$-negative. 
 Then
there exists a 
  bounded $F$-weight structure $w$ on $T=\lan H\ra$ in $\cu$ such that $\obj H\subset T_{w=0}$.

\end{enumerate}
\end{pr}
\begin{proof}

Assertions \ref{irdual} and \ref{irext} are immediate from Definition \ref{drwstr}.

The proof of assertions \ref{irffunctwd} and \ref{irsffunctwd} is similar to that of Proposition 1.5.1(1,2) 
of \cite{bws}. The axiom (iii) of relative weight structures yields that the composition morphism $F(w_{\le m}M)\to F(w_{\ge l+1}M')$ vanishes. Hence (an easy) Proposition 1.1.9 of \cite{bbd} yields the existence of (\ref{d2x3r1}).

Similarly, we obtain the existence of (\ref{d2x3r2}) if $m<l$. Moreover, any two distinct choices of $d$ (resp. of $c$) are easily seen (see the proof of loc. cit.) to differ by $s\circ f$ (resp. by $(f'\circ s)[-1]$) for some $s\in \cu(w_{\le m}M[1],w_{\ge l+1}M')$. 
Since $F(s)=0$ (by axiom (iii) of relative weight structures), we  conclude the proof of assertion \ref{irsffunctwd}.

The argument needed for the proof of assertion \ref{iwpt} is very similar to the one  used in the proof Theorem 2.2.1(11) of \cite{bger}.

We put $M'=M$, $l=j+1$, $m=j$ in assertion \ref{irsffunctwd}; this yields the existence of some set of $c_j$. Since $\cu_{w\ge 1-j}$ is extension-stable, it contains $\co (c_{-j})$ 
 (for any choice of the latter).
Completing the commutative triangle $M\stackrel{a_{-j}}{\to} w_{\ge -j}M\stackrel{c_{-j}}{\to} w_{\ge 1-j}M$ to an octahedral diagram (as was drawn in loc. cit.), we obtain that $\co (c_{j-1})$ is also a cone of some morphism $w_{\le -j-1}M[1]\to w_{\le -j}M[1]$. Since   $\cu_{w\le 1-j}$ is extension-stable also, we obtain 
assertion \ref{iwpt}.

(\ref{ibwpt}): If   $w_{\ge j_1+1}M=0$  (resp. $w_{\ge j_0}M=M$) then obviously 
  $M\in \cu_{w\ge j_0}$ (resp. $M\in \cu_{w\le j_1}$). Conversely, if $M\in \cu_{[j_0,j_1]}$, then nothing prevents us from choosing $w_{\ge j}M=0$ for all $j> j_1$ and $=M$ for all $j\le j_0$.

(\ref{irwe}): Certainly, if $G$ left (resp. right) weight-exact then $G(\cu_{w=0})\subset \cu'_{w'\le 0}$ (resp. $G(\cu_{w=0})\subset \cu'_{w'\ge 0}$).
Conversely, let $M\in \cu_{w\le 0}$ (resp. $M\in \cu_{w\ge 0}$). By the previous assertion, $M$ possesses a bounded  weight Postnikov tower with $M^i=0$ for $i<0$ (resp. for $i>0$). The structure of the tower yields that $M$ belongs to the envelope of $M^i[-i]$; this concludes the proof of the assertion.

The proof of assertion \ref{irgen} is similar to that of Theorem 4.3.2(II.1) of \cite{bws} (also, one can assume that $F=0$ here).
We take  the envelope of $H[i]$ for $i\le 0$ (resp. for $i\ge 0$) for $\cu_{w\le 0}$ (resp. for $\cu_{w\ge 0}$; see the Notation). Obviously,
 $\cu_{w\le 0}$ and $\cu_{w\ge 0}$ are Karoubi-closed, extension-stable, and satisfy the condition (i) of Definition \ref{drwstr}(I).
 $F$-orthogonality of $H$ easily yields conditions (ii) and (iii) of loc. cit. It remains to verify that any object of $\cu$ possesses a weight decomposition with respect to $w$.

 We define a certain notion of complexity for objects of $\cu$. For $M\in \obj H[i]$ (for some $i\in \z$) we will say that $M$ has complexity $\le 0$. If there exists a distinguished triangle $M\to N\to O$, and $M,O$ are of complexity $\le j$ for some $j\ge 0$ (they also could have smaller complexity) we will say that the complexity of $N$ is $\le j+1$. 
 Since any object of $\lan H\ra$ has finite complexity, it suffices to verify: for a distinguished triangle $M\to N\to O$ if $M,O$ possess weight decompositions (with respect to our $(\cu_{w\le 0},\cu_{w\ge 0}$)), then $N$ possesses a weight decomposition also.

 By assertion \ref{irsffunctwd}, we can complete the morphism $O[-1]\to M$ to a commutative square
 $$\begin{CD} (w_{\le 0}M)[-1]  @>{}>> O[-1] \\ @VV{}V @VV{}V \\
w_{\le 0} M @>{}>>M \\
\end{CD} $$
 Hence by the $3\times 3$-Lemma (i.e., Proposition 1.1.11 of \cite{bbd}) we can complete the distinguished triangle $M\to N\to O$ to a commutative  diagram
 \begin{equation}\label{dia3na3t}
\begin{CD}
w_{\le 0} M @>{}>>M @>{}>> w_{\ge 1}M \\
@VV{}V@VV{}V @VV{ }V\\ N'@>{}>>N
@>{}>>N''\\ @VV{}V@VV{}V @VV{}V\\
w_{\le 0} O@>{}>>O @>{}>> w_{\ge 1}O\\
\end{CD}
\end{equation}
 whose rows and columns are distinguished triangles (cf. Lemma 1.5.4 of \cite{bws}). We have $N'\in \cu_{w\le 0}$, $N''\in \cu_{w\ge 1}$ (by the definition of these classes). Hence $N$ possesses a weight decomposition indeed.

\end{proof}

\begin{rema}
One can also  glue relative weight structures similarly to Proposition \ref{pbw}(\ref{igluws}), and define weight structures for 'pure' localizations as in Proposition \ref{pbw}(\ref{iloc}). 
\end{rema}

As above, for (any) contravariant $H:\cu\to \au$ we denote $H\circ [-r]$ by $H^r$.

\begin{pr}\label{twreal}

I Let  
$H$ be a cohomological functor, $M\in \obj \cu$.
 Fix (any choice of) a bounded weight Postnikov tower for $M$ (see Proposition \ref{pbrw}(\ref{iwpt})).

1. There exists a {\it weight spectral sequence} $T$ with $E_1^{pq}(T)=H^q(M^{-p})\implies E_{\infty}^{p+q}(T)=H^{p+q}(M)$ (cf. Proposition \ref{pwss}(II.2)). 

2.  Denote the step of filtration given by ($E_1^{l,m-l}:$ $l\ge k$)
 on $H^{m}(M)$
by $W^{k}H^{m}(M)$. Then 
\begin{equation}\label{ewfil}
W^{k}H^{m}(M)=\imm (H^m(w_{\ge k}M)\to H^m(M)).
\end{equation}

3. Suppose that a not necessarily cohomological $H:\cu\to \au$ can be factored through $F$.  Then  
$W^{k}H^m(M)$ (defined 
by the formula (\ref{ewfil})) 
is $\cu$-functorial in $M$ (and does not depend on the choice of the corresponding weight decompositions).

II Let $F:\cu\to \du$, $F':\cu'\to \du$, and $G:\cu\to \cu'$ be  exact functors.
Let $w$ be an $F$-weight structure  for $\cu$, $w'$ be an $F'$-weight structure on $\cu'$; suppose that $G$ is weight-exact.

1. $G$ converts $w$-Postnikov towers into $w'$-Postnikov towers.

2. For 
any functor $H':\cu\to \au$ suppose that $H'$ can be factored through
$F'$. 
 Then 
we have $W^{k}H'^m(G(-))=W^{k}H^m(-)$ (here we define the weight filtration via (\ref{ewfil})).


\end{pr}
\begin{proof}

I 1,2: Immediate from the standard  properties of spectral sequences coming from  Postnikov towers; see the Exercises after \S IV.2 of \cite{gelman}.

3: This is an easy consequence of 
Proposition \ref{pbrw}(\ref{irffunctwd}) (cf. the proof of Proposition 2.1.2 of \cite{bws}).
Indeed,  
the right hand side square in the diagram (\ref{d2x3r1}) yields 
the existence of a commutative  diagram
\begin{equation}\begin{CD} \label{efcd}
H^m(w_{\ge k}M')@>{}>>H^m(M') 
\\
@VV{}V @VV{H^m(g)}V \\
H^m(w_{\ge k}M)@>{}>>H^m(M)  
\end{CD}\end{equation}
Now,
it suffices to verify  that for any $g\in \cu(M,M')$ and for any choice of $w_{\ge k}M,\ w_{\ge k}M'$ we have
$H^m(g)(\imm (H^m(w_{\ge k}M')\to H^m(M'))\subset \imm (H^m(w_{\ge k}M)\to H^m(M)$; the latter is an immediate consequence of 
the existence of (\ref{efcd}).


II 1. Obvious.

2. 
Immediate from assertions II.1 and I.3.

\end{proof}

\begin{rema}\label{rrwc}
1. Suppose that there exist $t$-structures $t_{\cu}$ for $\cu$ and $t_{\du}$ for $\du$ such that
$F$ is $t$-exact. Suppose also that for $M\in \cu^{t=0}$ there exists a choice of $w_{\le 0}M$ and  $w_{\ge 1}M$ belonging to $\hrt$. Then the morphism $F(M)\to F(w_{\ge 1} M)$ is
epimorphic in $\hrt_\du$. It follows: for the functor $H=H_0^{t_\du,op}\circ F$ the first level of the weight filtration of $H(F)=F(M)$ is just $F(w_{\ge 1} M)$. Here $H_0^{t_\du,op}$ is the zeroth homology with respect to $t$ with values in the category opposite to $\hrt_{\du}$ (we invert the arrow in order to make the functor cohomological).
So, such weight truncations are '$F$-functorial when they exist'; cf. Remark 1.5.2(2) and \S8.6 of \cite{bws}. Hence the corresponding weight filtrations are functorial also.

2. Unfortunately, it seems that 
weight spectral sequences given by 
the Proposition don't have to be canonical (in general). 

3. Without losing any generality, one can assume that $F=0$ in assertion II.2.

\end{rema}

\subsection{On mixed complexes of sheaves over  finite fields}\label{sperv}

Now let $S=X_0$ be a variety over a finite field $\fq$; let $X$ denote $X_0\times_{\spe \fq}\spe \ff$, where $\ff$ is the algebraic closure of $\fq$.
Let $F$ denote the extension of scalars functor $\dhsxz\to \dhsx$. We consider 
$\hetl:\dmc(X_0)\to \dhsxz$ (cf. \S\ref{smshgen}).

\begin{pr}
 1. The category $\dsh=\dbm(X_0,\ql)(\subset \dhsxz)$ of mixed complexes of sheaves (see \S5.1.5 of \cite{bbd})  can be endowed with an $F$-weight  structure $w_{\dsh}$
  such that $\dsh_{w_{\dsh}\le 0}=D^b_{\le 0}(X_0,\ql)$ (resp. $\dsh_{w_{\dsh}\ge 0}=D^b_{\ge 0}(X_0,\ql)$); here we use the notation from \S5.1.8 of ibid. 
 The heart of  $w_{\dsh}$ is the class of pure complexes of sheaves of weight $0$.

 2. $\hetl$ can be 
  factored through an exact $\mathbb{H}:\dmc(X_0)\to \dsh$; $\mathbb{H}$ is a  weight-exact functor (with respect to $\wchow$ and $w_{\dsh}$).

\end{pr}
\begin{proof}
1. Proposition 5.1.14 of \cite{bbd} yields all axioms of $F$-weight structures in our situation except the existence of weight decompositions.
So, by Proposition \ref{pbrw}(\ref{irgen}), it suffices to verify that the category of pure complexes of sheaves of weight $0$ (note that it is idempotent complete) generates $\dsh$. This is immediate from Theorem 5.3.5 of \cite{bbd}.

2. The same arguments as in  \S\ref{smshgen} (actually, a simplification of those, since we don't have to present $S$ as a limit) easily yield the existence of $\mathbb{H}$. We also obtain that it is weight-exact by 
Proposition \ref{pbrw}(\ref{irwe}). 

\end{proof}

\begin{rema}
1. In particular, we obtain that any object $M$ of $\dsh$ possesses a  weight Postnikov tower whose 'factors' are pure complexes of sheaves.

Besides, we obtain that for any (cohomological) $G:\dhsx\to \au$, $H=G\circ F\circ \hetl$,   
the Chow-weight filtration (see Remark \ref{rintel}) $(W^lH)(-)$ can be described as $(W^l_{w_{DSH}}G\circ F) (\hetl(-)) $ (for any $l\in \z$; see Proposition \ref{twreal}(II.2)). 

 Possibly, one could prove some more results of this sort via introducing a 'relative' analogue of the notion of {\it transversal} weight and $t$-structures (that was introduced in \S1 of \cite{btrans}); cf. also Remark \ref{rrwc}(1).

2. So, it is no surprise  that Theorem \ref{tfunctwchow} is a motivic analogue of the 'stability properties' 5.1.14 of \cite{bbd}.

\end{rema}

\appendix
\section{On the Chow $t$-structure and the virtual $t$-truncations}\label{ssupl} 

In \S\ref{sadjstr} we recall the notion of a $t$-structure adjacent to a weight structure (as introduced in \S4.4 of \cite{bws}).

In \S\ref{schowts} we use Theorem 4.5.2 of ibid. to prove the existence of the {\it Chow $t$-structure} on $\dms$ that is adjacent to the Chow weight structure for it (cf. Theorem \ref{twchow}(II)); we also establish certain functoriality properties of this $t$-structure (with respect to the motivic image functors, when $S$ varies).

In \S\ref{svirt} we recall the notion of {\it virtual $t$-truncations} (for cohomological functors from $\dmcs$), and relate virtual $t$-truncations with $\tchow$.

\subsection{On adjacent structures}\label{sadjstr} 

We recall the notion of adjacent weight and $t$-structures (that was introduced in \S4.4 of \cite{bws}). For $t$-structures we will use notation and conventions similar to those for weight structures in \S\ref{sws} (see also \S4.1 of \cite{bws}).
In particular, we will denote the heart of $t$ by $\hrt$ (recall that it is abelian); $\obj\hrt=\cu^{t=0}$.

We will say that  $t$ (for $\cu$) is non-degenerate if $\cap_{n\in \z}\cu^{t\le n}= \cap_{n\in \z}\cu^{t\ge n}=\ns$.

\begin{defi}\label{deadj} We say that a weight structure $w$ is
(left)
{\it adjacent} to a $t$-structure $t$ if $\cu_{w\ge 0}=\cu^{t\le 0}$.

\end{defi}

We will also need the following properties of adjacent structures.

\begin{pr}\label{padstr}

I Suppose that $\cu$
is  endowed with a
weight structure $w$  and also with an adjacent
$t$-structure $t$.

1. The functor $\cu(-,\hrt): \hrt \to \adfu(\hw^{op}, \ab)$ that sends $N\in
\cu^{t=0}$ to $M\mapsto \cu(M,N),\ (M\in \cu_{w=0}),$ is an  exact
embedding of $\hrt$ into the abelian category $\adfu(\hw^{op}, \ab)$.

2. Assume that $t$ is non-degenerate. Then
$\cu^{t=0}=\{M\in \obj \cu:\ \cu_{w=i}\perp M\ \forall \ i\neq 0
\}.$

II Moreover, let a triangulated category $\cu'$
be endowed with a
weight structure
 $w'$ and also with its adjacent
$t$-structure $ t'$. Let $F:\cu\to \cu'$
 be an exact functor.

1. $F$ is right weight-exact if and only if it is 
right  $t$-exact (i.e., if $F(\cu^{t\le 0})\subset \cu'^{t'\le 0}$).

2. Let $G:\cu'\to \cu$ be the right adjoint to $F$. Then $F$ is right (resp. left) weight-exact with respect to $w$ and $w'$ if and only if $G$ is 
left (resp. right) 
$t$-exact with respect to $t'$ and $t$.

III Let $\du\subset \cu$ be a full subcategory of compact objects endowed with a weight structure $w_{\du}$ (we denote its heart by $\hw_{\du}$). Suppose that $\cu$  admits arbitrary (small) coproducts and  that $\du$ weakly generates $\cu$. Then the following statements are valid.

1.  For the weight structure $w$ for $\cu$ given  by Proposition \ref{pbw}(\ref{iwegen}) there exists an adjacent $t$-structure; it is non-degenerate. 
$\hrt$ is isomorphic to $\adfu(\hw_{\du}^{op},\ab)$ (via the functor $N\mapsto (M\in \du^{\hw_{\du}=0}\mapsto \cu(M,N))$).

2. Suppose that $w_{\du'}$ and $\du'\subset \cu'$  satisfy the conditions for $w_{\du}$ and $\du\subset \cu$ described above; denote the corresponding adjacent  structures for $\cu$ by $w'$ and $t'$, respectively.

Let $F:\cu\to \cu'$ be an exact functor that maps $\du$ into $\du'$; suppose that it possesses a right adjoint $G$ that maps $\du'$ in $\du$.   Then the restriction of $F$ to $\du$ is  right (resp. left) weight-exact with respect to $w_{\du}$ and $w'_{\du'}$ if and only if $G$ is left (resp. right) 
$t$-exact with respect to $t'$ and $t$.

\end{pr}
\begin{proof}

I These are just parts 4 and 5 of Theorem 4.4.2 of \cite{bws}.

II.1. Immediate from the definition of adjacent structures.

2. See Remark 4.4.6 of ibid.

III 1. Immediate from Theorem 4.5.2 of ibid.

2. Immediate from the previous assertions by adjunction (we use the description of $\hrt$).
\end{proof}

\subsection{The Chow $t$-structure on $\dms$}\label{schowts}

Now we study the $t$-structure adjacent to $\wchowb$. 

\begin{pr}\label{ptchow}
I Let $S$ be an excellent separated finite-dimensional scheme. 

1. There exists a non-degenerate $t$-structure $\tcho(S)$ on $\dms$ that is adjacent to $\wchowb$ (the latter is given either by Theorem \ref{twchow}(II) or by Proposition \ref{pwchowa}(II)).

2. $\htcho(S)\cong \adfu(\chow(S){}^{op},\ab)$ (via the functor $N\mapsto (M\in \dmcs_{\wchow=0}\mapsto \dms(M,N))$).

II Let $f:X\to Y$ be a
 smoothly embeddable morphism of (excellent separated finite dimensional)  schemes. 

1. $f^!$ and $f_*$ are 
right  $\tchow$-exact (with respect to the corresponding Chow $t$-structures).

2. Suppose that $f$ is smooth. Then $f_*$ is (also) $\tchow$-exact.

\end{pr}
\begin{proof}
I Immediate from the definition of $\wchow$ and $\wchowb$, and  Proposition \ref{padstr}(I).

II The assertion follows easily  from Proposition \ref{pfunctwchowa} 
(or from Theorem \ref{tfunctwchow}(II) if $X$ and $Y$ are reasonable); see 
Proposition \ref{padstr}(III).

\end{proof}

\begin{rema}\label{rpure}

So, for any $N\in \obj \dms$ the Chow $t$-structure on $\dms$ allows us to 'slice' the cohomology theory $H:M\mapsto \dms(M,N)$, into 'pieces' $H^i:M\to \dms(M,\tchow_{=i}N)$; note that $H^i(N[j])=\ns$ for any $N\in \obj \chows\subset \obj \dmcs$, $j\neq i$ (see
Proposition \ref{padstr}(I.2)). One may call these $H^i$ {\it pure} cohomology theories.

 We will describe another (more general) method for slicing a cohomology theory into 'pure pieces' below; yet this reasoning does not demonstrate that the pieces of a representable cohomology theory are representable.
\end{rema}

\subsection{Virtual $t$-truncations with respect to $\wchow$; 'pure' cohomology theories}\label{svirt}

Now suppose that we are given an arbitrary cohomological functor $H:\dmcs\to \au$, $\au$ is an abelian category. Virtual $t$-truncations (defined in \S2.5 of \cite{bws} and studied in more detail in \S2 of \cite{bger}) allow us to 'slice' $H$ into pure pieces $H^i$. To this end we only use $\wchow$ (and have no need to put $H$ into some 'category of cohomological functors' $\dmcs\to \au$, and define a $t$-structure on this category). Virtual $t$-truncations also yield a functorial description of Chow-weight spectral sequences for cohomological functors (starting from $E_2$).

Now we just list the  main properties of virtual $t$-truncations (in the case where $(\cu,w)= (\dmcs,\wchow)$; the properties are the same as in the general case).

\begin{pr}
Let  $H:\dmcs\to \au$ and  $i\in\z$ be fixed. 

1. For any $M\in \obj\dmcx$ there exist unique morphisms $i_1(M)\in \dmcs(\wchow_{\le i-1} M,\wchow_{\le i}M)$ and $i_2(M)\in \dmcs(\wchow_{\ge i} M,\wchow_{\ge i+1}M)$
that 
fit into a commutative diagram
\begin{equation}\label{efwd}
 \begin{CD} \wchow_{\le i-1} M@>{}>>
M@>{}>> \wchow_{\ge i}M\\
@VV{i_1(M)}V@VV{\id_M}V@ VV{i_2(M)}V \\
\wchow_{\le i} M@>{}>>
M@>{}>> \wchow_{\ge i+1}M \end{CD}
\end{equation}
here the horizontal arrows are compatible with  (arbitrary fixed) weight decompositions of $M[i-j]$ (for $j=0,1$).

2. The correspondences $M\mapsto \imm H(i_1(M))$ and $M\mapsto \imm H(i_2(M))$ yield well-defined cohomological functors $\tau_{\ge 1-i}H,\,\tau_{\le -i-1}H:\dmcs\to \au$ (we call them virtual $t$-truncations of $H$).

3. For any $i\in \z$ the functor $\tau_{\le i}H$ kills $\dmcs_{\wchow\le -i-1}$; $\tau_{\ge i}H$ kills $\dmcs_{\wchow\ge 1-i}$.

4. $H$ yields naturally  an (infinite) sequence of transformations of functors
$$\dots \to (\tau_{\ge i+1}H)\circ [1]\to \tau_{\le i}H\to H\to  \tau_{\ge i+1}H\to (\tau_{\le i}H)\circ [-1]\to \dots$$
that yields a long exact sequence when applied to any $M\in \obj \dmcs$.

5. For any $j\in \z$ we have a natural isomorphism $\tau_{\le i}(\tau_{\ge j} H)\cong \tau_{\ge j}(\tau_{\le i}H)$.

6. We have a natural isomorphism $E^{-ii}_2(T(H,M)\cong \tau_{=i}H(=\tau_{\ge i}(\tau_{\le i}H))$ (see Proposition \ref{pwss}(II.2) 
 for the definition of $T(H,M)$ in this case). 

7. For $N\in \obj \dms$, $H=\dmc(-,N)$ we have: $\tau_{\le i}H\cong (-,\tchow_{\le i}N)$, $\tau_{\ge i}H\cong (-,\tchow_{\ge i}N)$, and $\tau_{= i}H\cong (-,\tchow_{=i}N)$.

\end{pr}
\begin{proof} Assertions 1--5 are immediate from Theorem 2.3.1 of \cite{bger}. Assertion 6 is contained in Theorem 2.4.2 of ibid. Assertion 7 follows from Proposition 2.5.4 of ibid.
 \end{proof}

\begin{rema}
1. Note that  $\tau_{=i}H$ vanishes on $\dmcs_{\wchow=j}$ for all $j\neq -i$; hence $\tau_{=i}H$ are 'pure' (cf. Remark \ref{rpure}).

2. One can also describe the whole $T(H,M)$ starting from $E_2$ in terms of (various) virtual $t$-truncations of $H$; see Theorem 2.4.2 of \cite{bger}.

3. Our definition of $\tau_{\ge i}H$ and $\tau_{\le i}H$ is compatible with the one of ibid. (if one replaces the cohomological notation for weights in ibid. with our current homological one), i.e., we do not change 'the signs' for virtual $t$-truncations. 
\end{rema}

\end{document}